\documentclass{amsart}

\usepackage[utf8]{inputenc}

\usepackage{amssymb}
\usepackage{amsmath}
\usepackage{amsthm}
\usepackage{enumitem}
\usepackage{pgf,tikz}
\usetikzlibrary{arrows}
\usepackage[center]{caption}

\numberwithin{equation}{section}

\theoremstyle{plain}
\newtheorem{thm}{Theorem}[section]
\newtheorem{lemma}[thm]{Lemma}
\newtheorem{prop}[thm]{Proposition}
\newtheorem{coroll}[thm]{Corollary}
\newtheorem{claim}[thm]{Claim}

\theoremstyle{definition}
\newtheorem{defn}[thm]{Definition}

\newtheorem{conj}[thm]{Conjecture}
\newtheorem{remark}[thm]{Remark}
\newtheorem{ex}[thm]{Example}
\newtheorem{prob}[thm]{Problem}

\newtheorem{innercustomthm}{Theorem}
\newenvironment{customthm}[1]
{\renewcommand\theinnercustomthm{#1}\innercustomthm}
{\endinnercustomthm}

\usepackage{xspace} % A makrok utani space korrekt kezelese miatt.

\newcommand{\R}{\mathbb{R}}
\newcommand{\Z}{\mathbb{Z}}

\newcommand{\cG}{\mathcal G}

\DeclareMathOperator{\conv}{Conv}

\newcommand{\ehr}{\mathrm{Ehr}}

\def\rank{{\rm rank}}

\def\minfas{{\rm minfas}}
\def\mindj{\nu}

\def\Arb{{\rm Arb}}
\def\Park{{\rm Park}}
\def\park{{\rm park}}
\def\tcQ{{\tilde{\mathcal{Q}}}}

\title{Degrees of interior polynomials and parking function enumerators}
\begin{document}

\author{Tam\'as K\'alm\'an}
\address{Department of Mathematics,
Tokyo Institute of Technology
%H-214, 2-12-1 Ookayama, Meguro-ku, Tokyo 152-8551 
and International Institute for Sustainability with Knotted Chiral Meta Matter (WPI-SKCM$^2$), Hiroshima University, Japan}
\email{kalman@math.titech.ac.jp}

\author{Lilla T\'othm\'er\'esz}
\address{MTA-ELTE Egerv\'ary Research Group and ELTE Eötvös Loránd University, P\'azm\'any P\'eter s\'et\'any 1/C, Budapest, Hungary}
\email{tmlilla@caesar.elte.hu}

%\keywords{root polytope, directed graph, $h^*$-vector, interior polynomial, ribbon structure, dissection, shelling order, greedoid\\

%MSC[2020]: 52B20; 52B22; 05C31; 05C62; 05C22}

\date{}

\begin{abstract}
    The interior polynomial of a directed graph is defined as the $h^*$-polynomial of the graph's (extended) root polytope, and it displays several attractive properties.
    Here we express its degree 
    in terms of the minimum cardinality of a directed join, and give a formula for the leading coefficient. We present natural generalizations of these results 
    to oriented regular matroids; in the process we also give a facet description for the extended root polytope of an oriented regular matroid.
    
    By duality, our expression for the degree of the interior polynomial implies a formula for the degree of the parking function enumerator 
    of an Eulerian directed graph (which is equivalent to the greedoid polynomial of the corresponding branching greedoid).
    We extend that result 
    to obtain the degree of the parking function enumerator of an arbitrary rooted directed graph in terms of the minimum cardinality of a certain type of feedback arc set. 
\end{abstract}

\maketitle

\section{Introduction}

In this paper we compute the degrees of two interrelated (in fact, in an appropriate sense, dual) graph and matroid polynomials. In particular, we show that they can be expressed using common graph/matroid theoretic concepts.

\subsection{Interior polynomials}

The first type of polynomial we deal with is the interior polynomial of a directed graph. 
This 
notion generalizes some well-studied graph invariants, most notably the specialization $T(x,1)$ of the Tutte polynomial, 
as well as
its extension to hypergraphs \cite{hiperTutte}.  
It is defined as the $h^*$-vector of the (extended) root polytope \cite{smooth_Fano}
\[\tcQ_G=\conv(\{\mathbf{0}\}\cup \{\, \mathbf1_h-\mathbf1_t\mid\overrightarrow{th}\in E\,\})\subset\mathbb R^V\]
associated to the digraph $G=(V,E)$. We denote the interior polynomial of $G$ by $I_G$.
When $G$ is a bidirected graph, $\tcQ_G$ is known as the \emph{symmetric edge polytope} of the underlying undirected graph.
Please see Section \ref{sec:interior} for detailed definitions. 
In this paper we show that the degree and leading coefficient
of the interior polynomial have meaningful connections to the graph structure. 

In a directed graph, we call a cut \emph{directed} if all of its edges point toward the same shore. An edge set in a digraph is called a \emph{directed join}, or \emph{dijoin} for short, if it intersects every directed cut. Our first main result is the following.

\begin{thm} \label{thm:degree_of_interior_poly}
	Let $G=(V,E)$ be a connected digraph.
	Then the degree of the interior polynomial of $G$ is equal to $|V|-1-\mindj(G)$, where $\mindj(G)=\min \{|K| \mid K \subseteq E \text{ is a dijoin of }G\}$.
\end{thm}

Recall that by a theorem of Lucchesi and Younger \cite{LucchesiYounger}, the quantity $\mindj(G)$ above is also the maximal number of edge-disjoint directed cuts in $G$. Furthermore, if the underlying undirected graph of $G$ is 2-edge connected, then $\mindj(G)$ is the minimum number of edges whose reversal yields a strongly connected orientation of $G$ \cite[Proposition 9.7.1]{Frank_book}.

We also express the leading coefficient of the interior polynomial $I_G$ of $G$.

\begin{defn}
	Let $G$ be a directed graph, and $F$ a subset of its edges. We call the vector $\mathbf z=\sum_{v\in V}z_v\mathbf1_v\in \mathbb{Z}^V$ with
	$$ z_v = |\{e \in F \mid v \text{ is the head of } e\}|- |\{e \in F \mid v \text{ is the tail of } e\}|$$
	the \emph{net degree vector} of $F$. 
\end{defn}

\begin{thm}\label{thm:leading_coeff_graph}
	The leading coefficient of $I_G$ equals the number of vectors that can be obtained as the net degree vector of a minimum cardinality dijoin in $G$.
\end{thm}

Let us remark that the constant and linear terms of interior polynomials can also be expressed by using basic Ehrhart theory.
First of all the constant term of any $h^*$-polynomial is $1$.
Furthermore, in all cases when $G$ is connected, the dimension of the extended root polytope is $|V|-1$. This is because, on the one hand, $\dim\tcQ_G\leq|V|-1$ is obvious from the definition; on the other hand, any spanning tree $T$ of $G$ gives rise to a simplex $\tcQ_T$ of dimension $|V|-1$ inside $\tcQ_G$. In fact, the extended root polytope can be triangulated by such simplices, which also happen to be unimodular. Hence $I_G$ agrees with the $h$-vector of the triangulation, and by the relation between $f$- and $h$-vectors,
the coefficient of 
the linear term
in $I_G$ is\footnote{Here $|E|+1$ is the number of generators of $\tcQ_G$, all of whom become vertices in the unimodular triangulation, and $|V|-1$ is the dimension.} 
\(|E|+1-(|V|-1)-1=|E|-|V|+1\), that is the genus (also known as the nullity, cyclomatic number, or simply the first Betti number) of $G$. 

\begin{remark}
\label{rem:higashitani}
In \cite{smooth_Fano}, Higashitani proved that the (extended) root polytope of a strongly connected digraph is reflexive, whence its interior polynomial has degree $|V|-1$ and it is palindromic. 
Thus for a strongly connected digraph, the leading coefficient is also $1$. We can recover these results from Theorems \ref{thm:degree_of_interior_poly} and \ref{thm:leading_coeff_graph}. Indeed, for a strongly connected digraph $G$ we have $\nu(G)=0$, and the unique net degree vector of minimal cardinality dijoins is the zero vector.    
\end{remark}

\begin{figure}
	\begin{tikzpicture}[scale=.60]
	\node [circle,fill,scale=.8,draw] (1) at (1,0) {};
	\node [circle,fill,scale=.8,draw] (2) at (2,1) {};
	\node [circle,fill,scale=.8,draw] (3) at (1,2) {};
	\node [circle,fill,scale=.8,draw] (4) at (0,1) {};
	\node [circle,fill,scale=.8,draw] (5) at (5,0) {};	
	\node [circle,fill,scale=.8,draw] (6) at (6,1) {};	
	\node [circle,fill,scale=.8,draw] (7) at (5,2) {};		
	\node [circle,fill,scale=.8,draw] (8) at (4,1) {};		
	\path [thick,<-,>=stealth] (1) edge [left] node {} (2);
	\path [thick,<-,>=stealth] (1) edge [above] node {} (4);
	\path [thick,<-,>=stealth] (3) edge [below] node {} (2);
	\path [thick,<-,>=stealth] (3) edge [left] node {} (4);
	\path [thick,->,>=stealth] (5) edge [above] node {} (1);
	\path [thick,->,>=stealth] (5) edge [below] node {} (8);
	\path [thick,->,>=stealth] (5) edge [above] node {} (6);
	\path [thick,->,>=stealth] (7) edge [below] node {} (6);
	\path [thick,->,>=stealth] (7) edge [left] node {} (8);
	\path [thick,->,>=stealth] (7) edge [left] node {} (3);
 \draw [gray] (3,-.5) -- (3,2.5);
 \draw [gray] (2,1) circle (.6);
  \draw [gray] (4,1) circle (.6);
  \draw [gray] (6,1.7) arc (90:270:.7);
  \begin{scope}[yscale=1,xscale=-1]
  \draw [gray] (0,1.7) arc (90:270:.7);
  \end{scope}
	\end{tikzpicture}
	\caption{A directed graph with five disjoint directed cuts.}
 \label{fig:bipartite_graph}
\end{figure}

\begin{ex}
\label{ex:kispelda}
    Figure \ref{fig:bipartite_graph} shows a directed graph $G$ whose interior polynomial is $I_G(x)=4x^2+3x+1$. This can be checked directly in various ways,  one of which is mentioned in the next Remark.
    Note that the 
    coefficient of $x$ 
    indeed agrees with $|E|-|V|+1=3$.
    As to the leading coefficient, first notice that the edge set of $G$ 
    is the disjoint union of five directed cuts, whence $\nu(G)\ge5$. In fact $\nu(G)=5$ since it is easy to find
    dijoins of cardinality $5$ (actually, there are $18$ of them). Thus Theorem \ref{thm:degree_of_interior_poly} confirms that $\deg I_G=8-1-5=2$ and by Theorem \ref{thm:leading_coeff_graph}, the leading coefficient should indeed be $4$ because the minimum cardinality dijoins induce the following four net degree vectors: 
    
\noindent
\parbox[c]{.23\linewidth}{
\begin{tikzpicture}[scale=.46]
\tikzstyle{o}=[circle,fill,scale=1.1,draw]
\begin{scope}[opacity=.1]	
 \node [o] (1) at (1,0) {};
	\node [o] (2) at (2,1) {};
	\node [o] (3) at (1,2) {};
	\node [o] (4) at (0,1) {};
	\node [o] (5) at (4.5,0) {};	
	\node [o] (6) at (5.5,1) {};	
	\node [o] (7) at (4.5,2) {};		
	\node [o] (8) at (3.5,1) {};
\end{scope}
\begin{scope}[opacity=.5]
	\path [thick,<-,>=stealth] (1) edge [left] node {} (2);
	\path [thick,<-,>=stealth] (1) edge [above] node {} (4);
	\path [thick,<-,>=stealth] (3) edge [below] node {} (2);
	\path [thick,<-,>=stealth] (3) edge [left] node {} (4);
	\path [thick,->,>=stealth] (5) edge [above] node {} (1);
	\path [thick,->,>=stealth] (5) edge [below] node {} (8);
	\path [thick,->,>=stealth] (5) edge [above] node {} (6);
	\path [thick,->,>=stealth] (7) edge [below] node {} (6);
	\path [thick,->,>=stealth] (7) edge [left] node {} (8);
	\path [thick,->,>=stealth] (7) edge [left] node {} (3);
 \end{scope}
 \node (11) at (1,0) {\tiny $1$};
 \node (12) at (2,1) {\tiny $-1$};
	\node (13) at (1,2) {\tiny $2$};
	\node (14) at (0,1) {\tiny $-1$};
	\node (15) at (4.5,0) {\tiny $-1$};	
	\node (16) at (5.5,1) {\tiny $1$};	
	\node (17) at (4.5,2) {\tiny $-2$};
	\node (18) at (3.5,1) {\tiny $1$};
 \end{tikzpicture}} %;
 \parbox[c]{.23\linewidth}{
 \begin{tikzpicture}[scale=.46]
\tikzstyle{o}=[circle,fill,scale=1.1,draw]
\begin{scope}[opacity=.1]	
 \node [o] (1) at (1,0) {};
	\node [o] (2) at (2,1) {};
	\node [o] (3) at (1,2) {};
	\node [o] (4) at (0,1) {};
	\node [o] (5) at (4.5,0) {};	
	\node [o] (6) at (5.5,1) {};	
	\node [o] (7) at (4.5,2) {};		
	\node [o] (8) at (3.5,1) {};
\end{scope}
\begin{scope}[opacity=.5]
 \path [thick,<-,>=stealth] (1) edge [left] node {} (2);
	\path [thick,<-,>=stealth] (1) edge [above] node {} (4);
	\path [thick,<-,>=stealth] (3) edge [below] node {} (2);
	\path [thick,<-,>=stealth] (3) edge [left] node {} (4);
	\path [thick,->,>=stealth] (5) edge [above] node {} (1);
	\path [thick,->,>=stealth] (5) edge [below] node {} (8);
	\path [thick,->,>=stealth] (5) edge [above] node {} (6);
	\path [thick,->,>=stealth] (7) edge [below] node {} (6);
	\path [thick,->,>=stealth] (7) edge [left] node {} (8);
	\path [thick,->,>=stealth] (7) edge [left] node {} (3);
 \end{scope}
 \node (11) at (1,0) {\tiny $2$};
 \node (12) at (2,1) {\tiny $-1$};
	\node (13) at (1,2) {\tiny $1$};
	\node (14) at (0,1) {\tiny $-1$};
	\node (15) at (4.5,0) {\tiny $-1$};	
	\node (16) at (5.5,1) {\tiny $1$};	
	\node (17) at (4.5,2) {\tiny $-2$};		
	\node (18) at (3.5,1) {\tiny $1$};		
 \end{tikzpicture}} %;
 \parbox[c]{.23\linewidth}{
  \begin{tikzpicture}[scale=.46]
\tikzstyle{o}=[circle,fill,scale=1.1,draw]
\begin{scope}[opacity=.1]	
 \node [o] (1) at (1,0) {};
	\node [o] (2) at (2,1) {};
	\node [o] (3) at (1,2) {};
	\node [o] (4) at (0,1) {};
	\node [o] (5) at (4.5,0) {};	
	\node [o] (6) at (5.5,1) {};	
	\node [o] (7) at (4.5,2) {};		
	\node [o] (8) at (3.5,1) {};
\end{scope}
\begin{scope}[opacity=.5]
 \path [thick,<-,>=stealth] (1) edge [left] node {} (2);
	\path [thick,<-,>=stealth] (1) edge [above] node {} (4);
	\path [thick,<-,>=stealth] (3) edge [below] node {} (2);
	\path [thick,<-,>=stealth] (3) edge [left] node {} (4);
	\path [thick,->,>=stealth] (5) edge [above] node {} (1);
	\path [thick,->,>=stealth] (5) edge [below] node {} (8);
	\path [thick,->,>=stealth] (5) edge [above] node {} (6);
	\path [thick,->,>=stealth] (7) edge [below] node {} (6);
	\path [thick,->,>=stealth] (7) edge [left] node {} (8);
	\path [thick,->,>=stealth] (7) edge [left] node {} (3);
 \end{scope}
 \node (11) at (1,0) {\tiny $1$};
 \node (12) at (2,1) {\tiny $-1$};
	\node (13) at (1,2) {\tiny $2$};
	\node (14) at (0,1) {\tiny $-1$};
	\node (15) at (4.5,0) {\tiny $-2$};	
	\node (16) at (5.5,1) {\tiny $1$};	
	\node (17) at (4.5,2) {\tiny $-1$};		
	\node (18) at (3.5,1) {\tiny $1$};		
 \end{tikzpicture}} %;
 \parbox[c]{.23\linewidth}{
  \begin{tikzpicture}[scale=.46]
\tikzstyle{o}=[circle,fill,scale=1.1,draw]
\begin{scope}[opacity=.1]	
 \node [o] (1) at (1,0) {};
	\node [o] (2) at (2,1) {};
	\node [o] (3) at (1,2) {};
	\node [o] (4) at (0,1) {};
        \node [o] (5) at (4.5,0) {};	
	\node [o] (6) at (5.5,1) {};	
	\node [o] (7) at (4.5,2) {};		
	\node [o] (8) at (3.5,1) {};
\end{scope}
\begin{scope}[opacity=.5]
 \path [thick,<-,>=stealth] (1) edge [left] node {} (2);
	\path [thick,<-,>=stealth] (1) edge [above] node {} (4);
	\path [thick,<-,>=stealth] (3) edge [below] node {} (2);
	\path [thick,<-,>=stealth] (3) edge [left] node {} (4);
	\path [thick,->,>=stealth] (5) edge [above] node {} (1);
	\path [thick,->,>=stealth] (5) edge [below] node {} (8);
	\path [thick,->,>=stealth] (5) edge [above] node {} (6);
	\path [thick,->,>=stealth] (7) edge [below] node {} (6);
	\path [thick,->,>=stealth] (7) edge [left] node {} (8);
	\path [thick,->,>=stealth] (7) edge [left] node {} (3);
 \end{scope}
 \node (11) at (1,0) {\tiny $2$};
 \node (12) at (2,1) {\tiny $-1$};
	\node (13) at (1,2) {\tiny $1$};
	\node (14) at (0,1) {\tiny $-1$};
	\node (15) at (4.5,0) {\tiny $-2$};	
	\node (16) at (5.5,1) {\tiny $1$};	
	\node (17) at (4.5,2) {\tiny $-1$};		
	\node (18) at (3.5,1) {\tiny $1$};		
 \end{tikzpicture}} .
\end{ex}

\begin{remark}
\label{rem:fst}
Let us explain a simple case that inspired Theorem \ref{thm:degree_of_interior_poly}.
The interior polynomial was first introduced by Kálmán \cite{hiperTutte}, for hypergraphs. By the main result of \cite{KP_Ehrhart}, that definition corresponds to the situation when $G=(U,W;E)$ is a bipartite graph with classes $U$ and $W$, and each edge is oriented toward $W$. 
(The graph $G$ of Example \ref{ex:kispelda} is of this type with $|U|=|W|=4$, $\nu(G)=5$, and $\deg I_G=2$. In particular, its interior polynomial may be evaluated by using \cite[Definition 5.3]{hiperTutte}.)

The hypergraphical approach to the interior polynomial $I_G$
easily yields
(see \cite[Proposition 6.1]{hiperTutte})
that (for graphs $G$ as above)
\begin{equation}
\label{eq:hyp}
\deg I_G\le\min\{|U|,|W|\}-1.
\end{equation}
%In that case, 
By considering star-cuts, it is also obvious that 
\begin{equation}
    \label{eq:fst}
    \nu(G)\ge\max\{|U|,|W|\}.
\end{equation}
In this special case, Theorem \ref{thm:degree_of_interior_poly} is the statement that the 
defects in 
the inequalities \eqref{eq:hyp} and \eqref{eq:fst} coincide. This realization, made jointly with Andr\'as Frank, predates Theorem \ref{thm:degree_of_interior_poly} and was subsumed by it later. Regrettably, neither \eqref{eq:hyp} nor \eqref{eq:fst} seems to have a generalization to 
wider classes of 
digraphs.
\end{remark}

We note that interior polynomials of hypergraphs have another, rather different generalization, this time to polymatroids instead of directed graphs (see \cite{hiperTutte}, as well as \cite{universal} 
for a two-variable version). As far as we know, the connection to $h^*$-polynomials does not extend to the polymatroid setting, and a degree formula for the interior polynomial of a polymatroid remains unknown.
We do not touch on polymatroids in this paper.
 
By analogy to directed graphs, one can introduce the (extended) root polytope of an oriented regular matroid and define the interior polynomial as its $h^*$-vector. 
(For details, see Section \ref{sec:interior_regular_matroids}.) 
It turns out that all arguments readily generalize to this case, and one is able to obtain analogous statements regarding the degree and leading coefficient.

\begin{thm} \label{thm:degree_of_interior_poly_matroids}
	Let $M$ be an oriented regular matroid of rank $r$.
	Then the degree of the interior polynomial $I_M$ of $M$ is 
	$r-\mindj(M)$, where $\mindj(M)=\min \{|K| \mid K  
	\text{ is a dijoin of }M\}$.
\end{thm}

\begin{thm}\label{thm:leading_coeff_matroid}
	Take an arbitrary totally unimodular representing matrix $A=[\mathbf a_i]_{i\in E}$ 
 for the oriented regular matroid $M$. The leading coefficient of $I_M$ equals the number of vectors that can be obtained as $\sum_{i\in K}\mathbf a_i$, where $K$ is a minimum cardinality dijoin of $M$.
\end{thm}

Even though Theorems \ref{thm:degree_of_interior_poly_matroids} and \ref{thm:leading_coeff_matroid} imply Theorems \ref{thm:degree_of_interior_poly} and \ref{thm:leading_coeff_graph}, respectively, we include separate proofs for the benefit of those readers who are less interested in oriented matroids. 
 A facet description of the extended root polytope is provided in Proposition \ref{prop:facet_description_matroid}. We remark that D'Al\`\i, Juhnke-Kubitzke and Koch \cite{DAli} found a similar facet description in a special case, namely the one that corresponds to the symmetric edge polytope of a graph. 
 
It is interesting to compare interior polynomials for different orientations of a given undirected graph. In an earlier paper \cite{sym_ribbon} we 
highlighted
a special case of this problem. Namely, 
we conjectured that for a bipartite graph, among its so-called semi-balanced orientations (see Section \ref{sec:interior}), the one in which all edges point toward the same class has a 
coefficientwise minimal interior polynomial. (There are of course two such orientations but their extended root polytopes are isometric and thus their interior polynomials identical.) 
Semi-balanced orientations are relevant because 
their root polytopes are the facets of the symmetric edge polytope of the bipartite graph.
Here we prove a weakened version of this conjecture, addressing only the degree: We show that, in fact among all orientations of a bipartite graph, the one with each edge pointing toward one of the classes has an interior polynomial of minimum degree. 
It remains an open problem which orientations minimize the degree of the interior polynomial among non-bipartite graphs.
On the other hand, in section \ref{sec:deg_min_orientations}
we show an example that among non-bipartite graphs, there does not always exist an orientation coefficientwise minimizing the interior polynomial.

The analogous questions regarding maxima are settled as follows: the degree of the interior polynomial is maximized exactly by the strongly connected orientations (at least as long as the graph is connected and without bridge edges, see section \ref{sec:deg_min_orientations}); co\-effi\-ci\-ent\-wise maximizing orientations do not always exist.

\subsection{Greedoid polynomials and parking function enumerators}

The other 
main object of this paper is the greedoid
polynomial, introduced by Björner, Korte, and Lov\'asz \cite{greedoid}. A special case of this notion is equivalent, via a change of variable, to parking function enumerators.
Greedoid polynomials are set up in such a way that interesting information is contained in their lowest-degree term. 
(The degree of the leading term depends only weakly on the greedoid, namely it is the size of the ground set minus the rank. For exact definitions, see Section \ref{sec:greedoid_and_parking}.)

In the special case of branching greedoids of Eulerian digraphs, the greedoid polynomial is related to the interior polynomial via duality.  Let us 
sketch this connection. For rooted directed graphs, the greedoid polynomial of the induced branching greedoid is equivalent to the enumerator of graph parking functions \cite{SweeHong_parking} (also commonly called $G$-parking functions or generalized parking functions \cite{postnikov-shapiro}) via a 
reversal of the sequence 
of the coefficients. 
Now, for Eulerian digraphs, the parking function enumerator agrees with the interior polynomial of the cographic matroid of the digraph \cite{Eulerian_greedoid}. 
Hence in this case the degree of the interior polynomial corresponds to the degree of the lowest term in
the greedoid polynomial. 

An edge set of a digraph is called a \emph{feedback arc set} if it intersects each directed cycle. Using this dual notion of a dijoin,
Theorem \ref{thm:degree_of_interior_poly_matroids} implies the following. 

\begin{thm}\label{thm:Eulerian_degree}
	The degree of the parking function enumerator of a connected Eulerian digraph $G$ (with any root) is equal to $|E(G)| - |V(G)| +1 - \minfas(G)$, where $\minfas(G)$ denotes the minimum cardinality of a feedback arc set in $G$.
	
	Equivalently, for 
	the branching greedoid of $G$ (with any root),
	the coefficient of $x^i$ in the greedoid polynomial is zero for $i=0, \dots, \minfas(G)-1$, and nonzero for $i=\minfas(G)$.
\end{thm}

We remark that, in connected Eulerian digraphs, not just the degree but each coefficient of the parking function enumerator/greedoid polynomial is independent of the root \cite{PP16,SweeHong_parking}. None of this holds for general digraphs.

We generalize Theorem \ref{thm:Eulerian_degree} to branching greedoids of all rooted digraphs and, in a certain sense, to all greedoids. (These settings do not correspond to interior polynomials anymore.) To state our result
on general 
directed branching greedoids, 
we define a rooted variant of a feedback arc set.
For a graph $G=(V,E)$ and edge set $F\subseteq E$, we put $G[F]=(V,F)$. For a digraph $G$ and vertex $s$, we call the graph \emph{$s$-root-connected} if each vertex is reachable along a directed path from $s$.

\begin{defn}\label{def:minfas(G,s)}
	Let $G$ be an $s$-root-connected digraph with edge set $E$. We say that a set of edges $F\subseteq E$ is an \emph{$s$-connected feedback arc set} if $G[E-F]$ is an $s$-root-connected acyclic digraph. (Such a set $F$ always exists.)
	We denote by $\minfas(G,s)$ the minimum cardinality of an $s$-connected feedback arc set of $G$.
\end{defn}

With that, our formula is as follows. 

\begin{thm}\label{thm:parking_enum_deg_general_digraph}
	Let $G=(V,E)$ be an $r$-root-connected 
 digraph.
	Then the degree of the parking function enumerator of $G$, rooted at $r$, is equal to $|E| - |V| +1 - \minfas(G,r)$.
	
	Equivalently, in the greedoid polynomial of the branching greedoid of $G$ rooted at $r$, the coefficients of $x^0, \dots, x^{\minfas(G,r)-1}$ are zero, and the coefficient of $x^{\minfas(G,r)}$ is nonzero. 
\end{thm}

\begin{remark}\label{rem:dir_graphs_constant_term}
    By \cite[Theorem 6.10]{greedoid},
    the constant term of the greedoid polynomial of a (root-connected) rooted digraph $G$ is zero if and only if $G$ contains a directed cycle.     Theorem \ref{thm:parking_enum_deg_general_digraph} strengthens this statement. Indeed, (for an $s$-root-connected digraph $G$) we have $\minfas(G,s) > 0$ if and only if $G$ contains a directed cycle.
\end{remark}

For general greedoids, we show that in order to determine the 
degree of the lowest term
of the greedoid polynomial, it is in fact enough to understand when a greedoid has nonzero constant term in its greedoid polynomial.

\begin{thm}\label{thm:greedoid_nonzero_terms}
	Let $X=\{E,\mathcal{F}\}$ be a greedoid of rank $r$, and for a subset $S\subseteq E$ let $X|_S$ denote the greedoid obtained by restricting $X$ to $S$. Let 
 \begin{multline*}
 k=\min \{|S|\mid S\subseteq E\text{ so that } \rank(X|_{E-S})=r\text{ and the constant term of the}\\ 
 \text{greedoid polynomial of $X|_{E-S}$ is nonzero}\}.
 \end{multline*}
	Then in the greedoid polynomial of $X$, the coefficient of $x^i$ is zero for $i=0, \dots, k-1$, and the coefficient of $x^k$ is nonzero.
\end{thm}

In fact, Theorem \ref{thm:parking_enum_deg_general_digraph} follows directly from Theorem \ref{thm:greedoid_nonzero_terms} and the characterization of Björner, Korte, and Lov\'asz mentioned in Remark \ref{rem:dir_graphs_constant_term}.

The paper is structured as follows.
Section \ref{sec:interior} contains our results on the interior polynomial in the most `classical' case of directed graphs, including the proofs of Theorems \ref{thm:degree_of_interior_poly} and \ref{thm:leading_coeff_graph}.
Section \ref{sec:deg_min_orientations} discusses the comparison of interior polynomials for different orientations of the same graph.
In Section \ref{sec:interior_regular_matroids} we generalize Section \ref{sec:interior} to oriented regular matroids and prove Theorems \ref{thm:degree_of_interior_poly_matroids} and \ref{thm:leading_coeff_matroid}.
Finally, in Section \ref{sec:greedoid_and_parking} we turn to parking function enumerators and greedoid polynomials, and prove
Theorems \ref{thm:Eulerian_degree}, \ref{thm:greedoid_nonzero_terms}, and \ref{thm:parking_enum_deg_general_digraph}.

\subsection{Acknowledgements}

We are grateful to András Frank for stimulating conversations and for pointing out to us Theorem 9.6.12 of his book \cite{Frank_book}.

TK was supported by consecutive Japan Society for the Promotion of Science (JSPS) Grants-in-Aid for Scientific Research C (nos.\ 17K05244 and 23K03108).

LT was supported by the National Research, Development and Innovation Office of Hungary -- NKFIH, grant no.\ 132488, by the János Bolyai Research Scholarship of the Hungarian Academy of Sciences, and by the ÚNKP-22-5 New National Excellence Program of the Ministry for Innovation and Technology, Hungary. This work was also partially supported by the Counting in Sparse Graphs Lendület Research Group of the Alfr\'ed Rényi Institute of Mathematics.

\subsection{Graph notation}
Let $G=(V,E)$ be a directed graph (abbreviated to digraph throughout). We allow loops and multiple edges.
For 
an edge set $S\subseteq E$, we denote by $G[S]$ the digraph with vertex set $V$ and edge set $S$.

By a \emph{subgraph}, in this paper we will mean 
an object as above.
That is, all vertices are automatically included, even if they become isolated in the subgraph. 
We call a subgraph $G[S]$ \emph{spanning} if it has the same number of connected components as $G$. (We note that the term spanning subgraph is sometimes used in a different sense. Our usage correlates with spanning sets in matroid theory.)
In particular for a subgraph of a connected graph, being \emph{connected} and \emph{spanning} mean the same thing. We will mostly use the latter term for trees and the former otherwise. 
%Many of our graphs will be directed and w
When we want to emphasize that connectedness is meant without regard to edge orientations, we will speak of \emph{weakly connected} (sub)graphs.

A non-empty set of edges $C^*\subseteq E(G)$ is called a \emph{cut} if there is a partition $V_0\sqcup V_1 =V(G)$ of the vertex set such that $C^*$ is the set of edges connecting a vertex of $V_0$ and a vertex of $V_1$. The sets $V_0$ and $V_1$ are called the \emph{shores} of the cut. Note that in a connected graph, a cut uniquely determines its two shores.
A cut $C^*$ is called \emph{elementary} if it is minimal among cuts with respect to containment; this is equivalent to the condition that
$E(G)-C^*$, as a subgraph, has exactly two (weakly) connected components.
A cut $C^*$ is \emph{directed} if either each edge of $C^*$ leads from $V_0$ to $V_1$, or each edge of $C^*$ leads from $V_1$ to $V_0$.

In a directed graph, we call a subgraph a \emph{spanning tree} if it is a spanning tree (i.e., connected and cycle-free) when forgetting the orientations. For a spanning tree $T$ and an edge $e\notin T$, the \emph{fundamental cycle} of $e$ with respect to $T$, denoted by $C(T,e)$, is the unique cycle in $T\cup e$. 
Similarly, for an edge $e\in T$, the \emph{fundamental cut} of $e$ with respect to $T$ is the unique (elementary) cut in the complement of $T-e$. We denote it by $C^*(T,e)$.

A set of edges $K\subseteq E(G)$ is called a \emph{directed join} or \emph{dijoin} if for any directed cut $C^*$ of $G$ we have $C^*\cap K \neq\emptyset$. We denote the minimum cardinality of a dijoin of $G$ by $\nu(G)$.
A set of edges $F\subseteq E(G)$ is called a \emph{feedback arc set} if for any directed cycle $C$ we have $C\cap F\neq\emptyset$. We denote the 
minimum cardinality of a feedback arc set of $G$ by $\minfas(G)$.

Let $G$ be a digraph with a fixed root vertex $s$. The graph $G$ is said to be \emph{$s$-root-connected} if each vertex is reachable along a directed path from $s$.
We call a subgraph %$A$ 
an \emph{arborescence} rooted at $s$ if 
it consists of an $s$-root-connected tree and some isolated points.
A \emph{spanning arborescence} is an arborescence 
without isolated points.
We will denote the set of spanning arborescences of $G$, rooted at $s$, by $\Arb(G,s)$.
We note that a digraph $G$, rooted at $s$, has a spanning arborescence (rooted at $s$) if and only if it is $s$-root-connected. 

\section{The degree of the interior polynomial of a digraph}%via classical Ehrhart theory}
\label{sec:interior}

Let $G=(V,E)$ be a directed graph. 
To an edge $e=\overrightarrow{th}\in E$, let us associate the vector $\mathbf{x}_e=\mathbf{1}_h-\mathbf{1}_t\in \mathbb{R}^V$. (Here $t,h\in V$ and $\mathbf{1}_t,\mathbf{1}_h\in\mathbb R^V$ are the corresponding generators.) The  
first of the following two notions
was introduced by Higashitani while studying smooth Fano polytopes \cite{smooth_Fano}.

\begin{defn}\label{def:root_polytope}
The \emph{root polytope} of a directed graph $G=(V,E)$ is the convex hull
\[\mathcal{Q}_G=\conv\{\,\mathbf{x}_e\mid e\in E\,\}\subset\mathbb R^V.\]
The \emph{extended root polytope} of $G$ is
\[\tilde{\mathcal{Q}}_G=\conv(\{\mathbf{0}\}\cup \{\, \mathbf{x}_e\mid e\in E\,\})\subset\mathbb R^V.\]
\end{defn}

Let us recall the following well known facts. 
See, for example, \cite[Lemma 3.5 and Corollary 3.6]{semibalanced} for details and proofs.

\begin{lemma}
    \label{lem:unimodular_simplex}
For $G=T$ a tree (or forest), both $\mathcal{Q}_T$ and $\tilde{\mathcal{Q}}_T$ are simplices, moreover, they are unimodular with respect to the lattice $\Z^V$. 
\end{lemma}

Since we are about to define 
the interior polynomial of a digraph
as the $h^*$-po\-ly\-nomial of its extended root polytope,
let us recall that
for any $d$-dimensional polytope $Q\subset\R^n$ with vertices in $\Z^n$, its \emph{$h^*$-polynomial} $\sum_{i=0}^d h^*_i t^i$, also commonly called the \emph{$h^*$-vector} of $Q$, is defined by Ehr\-hart's identity 
\begin{equation}
\label{eq:h-csillag}
\sum_{i=0}^d h^*_i t^i = (1-t)^{d+1} \ehr_Q(t),
\quad\text{where}\quad 
\ehr_Q(t)=\sum_{k=0}^\infty|(k\cdot Q)\cap\Z^n|\,t^k
\end{equation}
is the so called \emph{Ehrhart series} of $Q$. We note that $h^*_0=1$ whenever $d\ge0$, i.e., whenever $Q$ is non-empty.

Intuitively, the $h^*$-polynomial can be thought of as a refinement of volume. Indeed, $h^*(1)$ (that is, the sum of the coefficients) is equal to the normalized volume of the polytope, where by normalized we mean that the volume of a  $d$-dimensional unimodular simplex is $1$.

Now we are in a position to introduce our object of study for this section. 

\begin{defn}[Interior polynomial] %\cite{hiperTutte,KP_Ehrhart,semibalanced}
Let $G$ be a directed graph. We call the $h^*$-polynomial of the extended root polytope $\tilde{\mathcal{Q}}_G$ the \emph{interior polynomial} of $G$, and denote it with $I_G$.
\end{defn}

\begin{remark}
	In the earlier papers \cite{hiperTutte,KP_Ehrhart,semibalanced}, the interior polynomial was defined only for so-called semi-balanced digraphs, and as the $h^*$-polynomial of $\mathcal{Q}_G$ instead of $\tilde{\mathcal{Q}}_G$. As we will soon point out, for these graphs, the $h^*$-polynomials of $\mathcal{Q}_G$ and $\tilde{\mathcal{Q}}_G$ agree. Thus our current definition 
	contains the previous one.
\end{remark}

Let us repeat our main claim about the interior polynomial.

\begin{customthm}{\ref{thm:degree_of_interior_poly}}
	Let $G=(V,E)$ be a connected 
	digraph.
	The degree of the interior polynomial of $G$ is equal to $|V|-1-\mindj(G)$, where 
	\[\mindj(G)=\min \{|K| \mid K \subseteq E \text{ is a dijoin of }G\}.\]
\end{customthm}
Before proving Theorem \ref{thm:degree_of_interior_poly}, we examine root polytopes in more detail.
If $G$ is connected, then by \cite[Proposition 1.3]{smooth_Fano} and \cite[Corollary 3.2]{semibalanced} the dimension of $\mathcal{Q}_G$ is either $|V|-1$ or $|V|-2$. It is $|V|-2$ if and only if $G$ satisfies the following condition.

\begin{defn}\cite{semibalanced}\label{prop:semibalanced_has_layering}
	A directed graph $G$ is \emph{semi-balanced} if there is a function $\ell\colon V \to \Z$ such that we have $\ell(h)-\ell(t)=1$ for each edge $\overrightarrow{th}$ of $G$. We call such a function $\ell$ a \emph{layering} of $G$. 
\end{defn} 

We note that an alternative characterization of semi-balanced digraphs \cite[Theorem 2.6]{semibalanced} is that each cycle has the same number of edges going in the two directions around the cycle. This description was given as 
the definition in \cite{semibalanced}.

Now let us turn to the extended root polytope $\tilde{\mathcal{Q}}_G$. 
As an alternative way to think of it,
note that loop edges $f$ have $\mathbf{x}_f=\mathbf{0}$. Hence $\tcQ_G=\mathcal{Q}_{G\cup f}$, where $f$ is a loop edge attached to any vertex of $G$. As a digraph with a loop is never semibalanced, this implies the following.
\begin{prop}\label{prop:dimension}
    For a connected graph $G$, the dimension of $\tcQ_G$ is $|V|-1$.
\end{prop}

Note that if $G$ has a directed cycle $C$ then $\frac{1}{|C|}\sum_{e\in C}\mathbf{x}_e=\mathbf0$, that is, $\mathbf{0}$ is a point of $\mathcal{Q}_G$. Hence if $G$ is not acyclic, then $\mathcal{Q}_G=\tilde{\mathcal{Q}}_G$. 

On the other hand, if $G$ is semi-balanced with layering $\ell$, then $\mathcal{Q}_G$ lies in the affine hyperplane $\{\mathbf x\in\R^V\mid \ell\cdot\mathbf x=1\}$.
(Here we view the layering $\ell$ as a vector in $\R^V$ and use the standard dot product.) Thus, $\tcQ_G$ and $\mathcal{Q}_G$
do not coincide but as
$\tilde{\mathcal Q}_G$ is just a cone (of the minimum possible height)
over $\mathcal Q_G$, 
their $h^*$-vectors still agree. 

Hence if we defined the interior polynomial using $\mathcal{Q}_G$ instead of the extended root polytope, that would only make a difference for acyclic, but not semi-balanced graphs. However, 
in that class we do find examples which show that
Theorem \ref{thm:degree_of_interior_poly} would not hold if the $h^*$-polynomial of $\mathcal{Q}_G$ replaced that of $\tcQ_G$:

\begin{ex}\label{ex:h^*_for_acyclic_triangle}
    Consider the triangle $G=(\{v_1,v_2,v_3\}, \{\overrightarrow{v_1v_2},\overrightarrow{v_2v_3},\overrightarrow{v_1v_3}\})$. It is easy to check that $h^*_{\mathcal{Q}_G}(x)=1$, while $h^*_{\tcQ_G}(x)=x+1$. Moreover we have $\nu(G)=1$, whence $|V(G)|-1-\nu(G)=1$. Thus indeed, $h^*_{\mathcal{Q}_G}$ does not satisfy the degree formula, only $h^*_{\tcQ_G}$ does. 
\end{ex}

We will need a description of $\tilde{\mathcal{Q}}_G$ by linear inequalities (half-spaces). 
That requires the following notions.

\begin{defn}
\label{def:functional_of_cut}
Let $C^*$ be a cut in the graph $G$ with shores $V_0$ and $V_1$. Let $f_{C^*}$ be the linear functional with $f_{C^*}(\mathbf{1}_v)=1$ when $v\in V_1$ and $f_{C^*}(\mathbf{1}_v)=0$ when $v\in V_0$. 
If $G$ is directed and $C^*$ is a directed cut, we will always suppose that $V_1$ is the shore containing the heads of the edges in the cut. 
We will refer to $f_{C^*}$ as the \emph{functional induced by the cut} $C^*$.
\end{defn}

\begin{defn}
\label{def:admissible_layering}
Let $G=(V,E)$ be a directed graph.
A function $\ell\colon V \to \mathbb{Z}$ is called an \emph{admissible layering} of $G$ if $\ell(h)-\ell(t)\leq 1$ holds for each $\overrightarrow{th}\in E$, and the edges with $\ell(h)-\ell(t)= 1$ form a (weakly) connected subgraph of $G$.
\end{defn}

It is not hard to show
that any admissible layering of a semi-balanced graph is actually a layering in the sense of Definition \ref{prop:semibalanced_has_layering}, so we have not introduced anything new into that context. 
We will not rely on this fact later.

Then, our description is as follows.

\begin{prop}\label{prop:supporting_hyperplanes_of_Q_G}\label{prop:extendedrootpoly}
The extended root polytope of any connected digraph $G$ satisfies
	\begin{equation}
 \label{eq:extendedrootpoly}
 \tilde{\mathcal{Q}}_G = \left\{\mathbf x \in \mathbb{R}^V \middle|
	\begin{array}{cl}
	f_{C^*}(\mathbf x) \geq 0 & \text{for all elementary
	directed cuts }C^*\text{ of }G\\
	\ell\cdot\mathbf x \leq 1& \text{for all admissible layerings $\ell$ of $G$}\\
	\mathbf{1}\cdot\mathbf x = 0& 
 \end{array}\right\}.
 \end{equation}
\end{prop}

Here $\mathbf 1=\sum_{v\in V}\mathbf 1_v\in \R^V$. 
It will be clear from the proof below that $f_{C^*}(\mathbf x) \geq 0$, whenever $\mathbf x\in \tilde{\mathcal{Q}}_G$, holds for non-elementary directed cuts $C^*$, too.

Note that if $\ell$ is an admissible layering then so is $\ell+m\cdot\mathbf 1$ for all $m\in\Z$. As is obvious from the proposition, members of such equivalence classes of admissible layerings describe the same constraint for $\tilde{\mathcal{Q}}_G$. If we chose a representative of each class, for example by assuming that $\ell(v_0)=0$ for some fixed vertex $v_0$, then the number of constraints would become formally finite. (Indeed, there are only finitely many ways to choose the connected subgraph whose edges $\overrightarrow{th}$ satisfy $\ell(h)-\ell(t)=1$, and with $\ell(v_0)=0$ that subgraph already determines $\ell$.)

\begin{coroll}
\label{cor:facets}
For any connected digraph $G$, the facets of $\tilde{\mathcal{Q}}_G$ are as follows: 
 \begin{enumerate}
     \item \label{partone}
     $C^*\mapsto\{\mathbf x\in\tilde{\mathcal{Q}}_G\mid f_{C^*}(\mathbf x)=0\}$ gives a one-to-one correspondence between the elementary directed cuts of $G$ and facets containing $\mathbf0$. 
     \item \label{parttwo}
     $\ell\mapsto\{\mathbf x\in\tilde{\mathcal{Q}}_G\mid\ell\cdot\mathbf x=1\}$ induces a bijection between the equivalence classes of admissible layerings described above and facets not containing $\mathbf0$.
 \end{enumerate}
\end{coroll}

Proposition \ref{prop:supporting_hyperplanes_of_Q_G} contains some previously published special cases, such as \cite[Proposition 3.6]{KP_Ehrhart} and \cite[Theorem 3.1]{arithm_symedgepoly}.
A facet description of $\mathcal Q_G$, very similar to Corollary \ref{cor:facets}, was recently obtained by Numata et al.\ \cite{numata}.

\begin{proof}[Proof of Proposition \ref{prop:supporting_hyperplanes_of_Q_G}]
    First note that each vector $\mathbf x\in\tcQ_G$ satisfies the conditions of the right hand side of \eqref{eq:extendedrootpoly}.
    Clearly $\mathbf{1}\cdot \mathbf{x}_e = 0$ for each $e\in E(G)$ and $\mathbf{1}\cdot \mathbf{0} = 0$, whence $\mathbf{1}\cdot \mathbf{x} = 0$ for each $\mathbf{x}\in \tilde{\mathcal{Q}}_G$.
    Similarly for any admissible layering $\ell$, by definition we have $\ell \cdot \mathbf{x}_e \leq 1$ for each $e\in E$, which in addition to the obvious $\ell \cdot \mathbf{0}=0$ implies $\ell\cdot\mathbf x \leq 1$ for each $\mathbf x \in \tcQ_G$.
    Finally let $C^*$ be any directed cut. Then by definition, the induced functional $f_{C^*}$ is such that $f_{C^*}(\mathbf x_e)=0$ for $e\notin C^*$ and $f_{C^*}(\mathbf x_e)=1$ for $e\in C^*$; moreover $f_{C^*}(\mathbf{0})=0$. Hence each $\mathbf x\in \tcQ_G$ satisfies $f_{C^*}(\mathbf x)\geq 0$.

	Conversely, we show that 
	any element $\mathbf x$ of
	%a vector satisfying 
	the right hand side belongs to $\tcQ_G$.
Proposition \ref{prop:dimension} claimed that $\dim\tcQ_G=|V|-1$. The condition $\mathbf{1}\cdot\mathbf x = 0$ says that $\mathbf x$ is in the same hyperplane of $\R^V$ as $\tcQ_G$; in the rest of the proof we work relative to this hyperplane.
   
   Consider an arbitrary facet $F$ of $\tcQ_G$.
   It suffices to show that $\mathbf x$ lies either in the hyperplane of $F$ or on the same side of the hyperplane as the interior of $\tcQ_G$.
   
   The facet $F$ needs to contain $|V|-1$ affine independent vertices of $\tcQ_G$. If $F$ does not contain $\mathbf{0}$, then those are $|V|-1$ affine independent vertices of type $\mathbf x_e$, which means that the corresponding edges $e$ form a spanning tree $T$. The tree $T$, along with the orientations of its edges, determines a layering $\ell\colon V\to\Z$ so that 
   $\ell\cdot\mathbf x_e=1$  for all $e\in T$. This layering is necessarily admissible, for otherwise there would be an edge $f$ with $\ell\cdot\mathbf x_f\ge2$, 
  meaning that the vertices $\mathbf 0$ and $\mathbf x_f$ lie on opposite sides of the hyperplane of $F$. 
  Therefore the hyperplane of the facet is $\{\mathbf x \mid \ell\cdot\mathbf x=1\}$ for the admissible layering $\ell$, and as $\ell\cdot\mathbf{0}=0$, the polytope $\tilde{\mathcal{Q}}_G$ lies in the half-space $\{\mathbf p\mid \ell\cdot\mathbf p\leq 1\}$, whence indeed the hyperplane of $F$ does not separate $\mathbf{x}$ from $\tcQ_G$. 
   
   If $F$ contains $\mathbf{0}$, then it must additionally contain $|V|-2$ affine independent vectors of the form $\mathbf x_e$. Here the corresponding edges $e$ form a two-component forest in $G$, whose components give rise to an elementary cut $C^*$. That in turn induces the functional $f_{C^*}$ that vanishes along $F$. This and $F$ being a facet imply that $C^*$ is a directed cut, for otherwise there would exist vectors $\mathbf x_e$ on both sides of the kernel of $f_{C^*}$. As the vectors $\mathbf{x}_e$ for $e\in C^*$ have $f_{C^*}(\mathbf{x}_e) = 1$, clearly $\tcQ_G$ lies in the half-space $\{\mathbf p \mid f_{C^*}(\mathbf p)\geq 0\}$. Thus again, the hyperplane of $F$ does not separate $\mathbf x$ from $\tcQ_G$.
   \end{proof}

   \begin{proof}[Proof of
   Corollary \ref{cor:facets}] This is a consequence of the following additions to the last two paragraphs of the previous proof. For any admissible layering $\ell$, the corresponding connected subgraph of $G$ contains a spanning tree and thus the supporting hyperplane $\{\mathbf x \mid \ell\cdot\mathbf x=1\}$ contains $|V|-1$ affine independent vertices of $\tcQ_G$. 
   Hence we have a surjection from admissible layerings to facets not containing $\mathbf0$. Since the facet corresponding to $\ell$ has dimension $|V|-2$, its orthogonal complement with respect to $\R^V$ is $2$-dimensional and clearly spanned by $\ell$ and $\mathbf1$.
   
   For an elementary directed cut $C^*$, the supporting hyperplane $\{\mathbf x\mid f_{C^*}(\mathbf x)=0\}$ contains a facet through $\mathbf0$ because $E-C^*$ contains a two-component forest. Therefore the map given in \eqref{partone} of the Corollary makes sense; it is surjective by the previous proof; finally it is also injective because $C^*$ can be recovered as the set of edges $e$ so that $\mathbf x_e$ is not in the facet.
   \end{proof}
   
Note that we can also conclude from the previous proof that (within the subspace $\{\mathbf x\mid \mathbf 1 \cdot \mathbf x = 0\}$), the hyperplanes $\{\mathbf x \mid f_{C^*}(\mathbf x)= 0\}$ for elementary directed cuts $C^*$ and the hyperplanes $\{\mathbf x\mid \ell\cdot\mathbf x =1\}$ for admissible layerings $\ell$ all contain facets of $\tcQ_G$, and collectively these are all the facets of $\tcQ_G$. 

A key ingredient in proving Theorem \ref{thm:degree_of_interior_poly} will be the following corollary of Ehrhart--Macdonald reciprocity.

\begin{thm}\cite[Theorem 4.5]{BeckRobbins}\label{thm:degree_of_h*_BR}
	Let $P\subset\R^n$ be a $d$-dimensional 
	($d\ge0$)
	lattice polytope with $h^*$-polynomial $h^*_d t^d + \dots + h^*_1 t + 1$. 
	Then $h^*_d=\dots=h^*_{k+1}=0$ and $h^*_k\neq 0$ if and only if $(d-k+1)P$ is the smallest integer dilate of $P$ that contains a lattice point in its relative interior.
\end{thm}

Notice here that $(d+1)P$ certainly contains an interior lattice point. The degree of the $h^*$-polynomial of $P$ tells us exactly `how much sooner' such a point occurs.

In our cases, $\tilde{\mathcal{Q}}_G$ is a $(|V|-1)$-dimensional polytope by Proposition \ref{prop:dimension}. 
Thus if we show that \[(\mindj(G)+1)\cdot\tilde{\mathcal{Q}}_G\] is the smallest integer dilate of the extended root polytope that contains a lattice point in its interior, then by Theorem \ref{thm:degree_of_h*_BR}, it follows that the degree of $I_G$ is indeed $|V|-1-\mindj(G)$. For this, we need to clarify the connection between dijoins of $G$ and interior points of $\tcQ_G$.

First, let us make a simple observation about dijoins. 
Let $C$ be any cycle in $G$. We call $C$ a \emph{signed cycle} if we have picked one of its cyclic orientations as ``positive.'' Then we can partition $C$ to $C^-\sqcup C^+$ where $C^-$ is the set of edges pointing in the negative cyclic direction, and $C^+$ is the set of edges pointing in the positive cyclic direction. Signed cycles are not to be confused with directed cycles; the latter means a cycle whose edges all point in the same cyclic direction. We will need two simple lemmas on the relationship of cycles and minimal dijoins.

\begin{lemma}\label{lem:min_dijoin_cyclefree}
    If $C$ is a cycle and $K$ is a dijoin such that $C\subseteq K$, then $K$ is not minimal (with respect to inclusion).
\end{lemma}

\begin{proof}
    Choose an orientation $C=C^+\sqcup C^-$ of $C$
    %. Either $C^+$ or $C^-$ is nonempty, suppose 
    so that $C^- \neq \emptyset$. We show that $K-C^-$ is also a dijoin.
    Indeed if a directed cut does not intersect $C$, then it is covered by $K-C$ and hence by $K-C^-$.
    Else if a directed cut intersects $C$, then it intersects both $C^-$ and $C^+$.
    %Hence, any directed cut that intersects $C$ 
    Thus it 
    is also covered by $K-C^-$.
\end{proof}

\begin{lemma}\label{lem:min_dijoin_long_arc_free}
Suppose that $C$ is a 
signed cycle with a choice of orientation such that $|C^-|>|C^+|$.
If a dijoin $K$ contains $C^-$, then $K$ is not a minimum cardinality dijoin.
\end{lemma}

\begin{proof}
This follows from $|K-C^-\cup C^+|<|K|$ and $K-C^-\cup C^+$ also being a dijoin. That is so because if a directed cut does not intersect $C$, then it is covered by $K-C^-\cup C^+$ the same way as by $K$; if a directed cut intersects $C$ then it intersects both $C^-$ and $C^+$; in particular it is covered by $K-C^-\cup C^+$.
\end{proof}

Let us call a dijoin $K$ a \emph{long arc dijoin} if it contains $C^-$ for some signed cycle $C$ with $|C^-|>|C^+|$. (Note that $C^-$ and $C^+$ need not be sequences of consecutive edges.) 
Lemma \ref{lem:min_dijoin_long_arc_free} implies that long arc dijoins never have minimum cardinality among dijoins.

We claim the following characterization of relative interior points of $\tcQ_G$. The distinction between $<1$ and $\leq1$ is a crucial subtlety in the statement. Note that since $\mathbf0$ is a generator (although not always a vertex) of $\tcQ_G$, our formulas can be thought of as convex linear combinations even when the sum of the coefficients is less than $1$.

\begin{prop}\label{prop:interior_point_iff_dijoin}
	A point $\mathbf p\in \tilde{\mathcal{Q}}_G$ is in the relative interior of $\tcQ_G$ 
 if and only if there exists a dijoin $K$ of $G$, whose underlying undirected graph is cycle-free, such that
 $\mathbf p = \sum_{e\in K} \lambda_e \mathbf{x}_e$
 with
 $\lambda_e > 0$ for each $e\in K$ and $\sum_{e\in K}\lambda_e < 1$.
\end{prop}

\begin{proof}
    By Proposition \ref{prop:supporting_hyperplanes_of_Q_G}, a point $\mathbf p\in \tcQ_G$ is in the interior of $\tilde{\mathcal{Q}}_G$ if and only if $f_{C^*}(\mathbf{p})> 0$ for each directed cut $C^*$ and $\ell\cdot \mathbf{p} < 1$ for each admissible layering $\ell$. 
    Recall that the functional induced by $C^*$ satisfies 
    $f_{C^*}(\mathbf{x}_e)=0$ whenever $e\notin C^*$, and $f_{C^*}(\mathbf{x}_e)=1$ for each $e\in C^*$. 
    
    We start with the `if' direction.
    Suppose that $K$ is a %cycle-free 
    dijoin and $\mathbf{p}=\sum_{e\in K} \lambda_e \mathbf{x}_e$, with $\sum_{e\in K}\lambda_e < 1$ and $\lambda_e>0$ for each $e\in K$. 
    (The cycle-freeness of $K$ is not needed for this direction.)
    Then $\mathbf{p}=\sum_{e\in K} \lambda_e \mathbf{x}_e + (1-\sum_{e\in K}\lambda_e)\cdot\mathbf{0}$, in particular $\mathbf{p}\in \tilde{\mathcal{Q}}_G$.
    For any directed cut $C^*$, we have $f_{C^*}(\mathbf{p})=\sum_{e\in C^*\cap K} \lambda_e > 0$, since the intersection is nonempty (by the definition of a dijoin), and the summands are all positive. 
    Similarly, for any admissible layering $\ell$ we have $\ell\cdot \mathbf{x}_e\leq 1$ for any $e\in E$, and $\ell\cdot \mathbf{0}=0$. Since the coefficient of $\mathbf{0}$ is positive, this implies $\ell\cdot\mathbf{p} < 1$.
    Thus $\mathbf{p}$ is in the interior of $\tcQ_G$.
	
    Now we turn to the `only if' direction.
    Let $\mathbf{p}$ be an arbitrary point in the relative interior of $\tcQ_G$, and take any convex combination $\mathbf{p}=\sum_{e\in S} \lambda_e \mathbf{x}_e + \mu \cdot\mathbf{0}$ with $\lambda_e>0$ for all $e\in S$, where $S\subseteq E$ is some subset. For a point of $\tcQ_G$, we can always find such a formula (usually more than one). We claim that if $\mathbf{p}$ is in the interior, then $S$ is necessarily a dijoin. Indeed, suppose that $S$ is disjoint from a directed cut $C^*$. Then $f_{C^*}(\mathbf{p})=\sum_{e\in S}\lambda_e \cdot f_{C^*}(\mathbf{x}_e)+\mu\cdot f_{C^*}(\mathbf{0})=0$, which contradicts $\mathbf{p}$ being an interior point of $\tcQ_G$.

Next we show that either $\sum_{e\in S} \lambda_e < 1$ or $S$ is a long-arc dijoin. 
After that, we will establish that if $S$ is a long-arc dijoin, then there exists another dijoin $S'$ such that $\mathbf{p}=\sum_{e\in S'} \lambda'_e\mathbf{x}_e$ 
and $\sum_{e\in S'} \lambda'_e < 1$.
Finally, we will show how to remove edges from the dijoin
to ascertain cycle-freeness.

So suppose that the dijoin $S$ we found for $\mathbf p$ is such that
$\sum_{e\in S} \lambda_e = 1$. 
Assume for a contradiction that $S$ is not a long-arc dijoin, that is, for all %cycles
signed cycles $C$  so that $|C^-|>|C^+|$, it does not contain $C^-$. We will reach a contradiction by showing that $\mathbf{p}$ is on the boundary of $\tcQ_G$. 
More concretely, we will find a vector 
$\ell$ such that $\ell\cdot \mathbf{x}_e=1$ for each $e\in S$ and $\ell\cdot \mathbf{x}_e\leq 1$ for each $e\notin S$. The existence of such a vector follows easily from the Farkas lemma and we will rely on that argument in Section \ref{sec:interior_regular_matroids} in the regular matroid case. Here however, let us construct $\ell$ in a more elementary way.

Choose a spanning forest within $G[S]$ and use it to fix integer values $\ell(v)$ for each $v\in V$ in such a way that $\ell\cdot \mathbf{x}_e=1$ for each edge $e$ in the forest. This can be done uniquely, up to one additive constant for each connected component of $G[S]$. For other edges of $S$ within one component, $\ell\cdot \mathbf{x}_e=1$ is automatically satisfied because they make fundamental cycles $C\subseteq S$ with the forest and such cycles have to have $|C^-|=|C^+|$ by our assumption that $S$ is not a long-arc dijoin.

It remains to ensure that $\ell\cdot\mathbf{x}_e\leq 1$ for edges $e\in E-S$. For some choice of $\ell$ as above, let us consider the set of `bad' edges that violate this condition,  that is 
\[B_\ell=\{e\in E-S\mid \ell\cdot\mathbf x_e>1\},\text{ 
as well as the set }
A_\ell=\{e\in E-S\mid \ell\cdot\mathbf x_e=1\}\]
where a violation `almost' occurs.
We use $B_\ell\cup A_\ell$ as the set of edges in an auxiliary graph $G'$, in which the vertices are the connected components of 
$G[S]$.
(I.e., $G'$ is the graph obtained by contracting each connected component of  $G[S]$
to a point and only keeping elements of $B_\ell$ and $A_\ell$ as edges, some of which may become loops.)

Suppose that $C'$ is a directed cycle in $G'$. Then there exists a pullback cycle $C$ in $G$ that is obtained by adding edges of $S$ to $C'$. As $C'$ is directed, 
$C$ can be oriented in such a way
that $C'\subseteq C^+$, and hence $C^-\subseteq S$. Because $\ell$ jumps by at least $1$ along each edge of $C'$ and by exactly $1$ along edges of $S$, we have to have $|C^-|\geq|C^+|$. But since $C^-\subseteq S$ and $S$ is not a long-arc dijoin, only $|C^-|=|C^+|$ is possible and that implies $C'\subseteq A_\ell$. 

Now as elements of $B_\ell$ cannot be in a directed cycle in $G'$, they all must be in directed cuts of $G'$. For any such cut, we may lower the values of $\ell$ by $1$ over all connected components of 
$G[S]$
%$S$ 
that form the head-shore of the cut. It is clear from the construction that this new vector $\ell'$ has $B_{\ell'}\subseteq B_\ell$, moreover
\[\sum_{e\in B_{\ell'}}\ell'\cdot\mathbf x_e<\sum_{e\in B_\ell}\ell\cdot\mathbf x_e.\]

We may repeat our modification step for $\ell'$ instead of $\ell$, and so on as long as $B$ is non-empty, but the last observation ensures that the process will be finite. At the end of it $B$ does become empty and we have found the desired vector $\ell$ such that $\ell \cdot \mathbf{x}_e \leq 1$ for $e\in E-S$ and $\ell \cdot \mathbf{x}_e = 1$ for $e\in S$. This contradicts the assumption that $\mathbf p$ is an interior point. Hence we have proved that $\sum_{e\in S} \lambda_e = 1$ implies that $S$ is a long-arc dijoin.

Next, we show that if $\mathbf{p}=\sum_{e\in S} \lambda_e \mathbf{x}_e$ with $\lambda_e>0$ for all $e\in S$ and $\sum_{e\in S} \lambda_e = 1$, where $S$ is a long-arc dijoin, then we can find a different dijoin $S'$ such that $\mathbf{p}=\sum_{e\in S'} \lambda'_e \mathbf{x}_e$ with $\lambda'_e>0$ for all $e\in S'$ and $\sum_{e\in K} \lambda'_e < 1$.

Let $C$ be a 
signed cycle that has $C^-\subseteq S$ and $|C^-|>|C^+|$. Such a cycle exists because of the definition of a long-arc dijoin.
Let $\delta = \min \{\lambda_e \mid e\in C^-\}$ and let $M=\{ e\in C^-\mid\lambda_e=\delta\}$. Then, by the obvious $\sum_{e\in C^-}\mathbf x_e=\sum_{e\in C^+}\mathbf x_e$, we have
$\mathbf p = \sum_{e\in S-M\cup C^+} \lambda'_e \mathbf{x}_e + \mu' \cdot \mathbf{0}$ where 
$$\lambda'_e=\left\{\begin{array}{cl} 
\lambda_e & \text{if $e\notin C$},  \\
\lambda_e-\delta & \text{if $e \in C^-$}, \\
\lambda_e+\delta & \text{if $e \in C^+$},
\end{array} \right.$$
and $\mu'=(|C^-|-|C^+|)\cdot\delta$. This is a new convex combination for $\mathbf p$, with positive coefficients.
Moreover, $S-M\cup C^+$ is also a dijoin because any directed cut that intersects $C$ necessarily intersects both $C^+$ and $C^-$. Hence we have found a linear combination for $\mathbf{p}$ using a dijoin and $\mathbf0$, such that 
the coefficient of $\mathbf{0}$ is positive. 

It is left to show that $S$ can be chosen cycle-free. 
If there is any cycle $C\subseteq S$, then choose its orientation so that $|C^-|\geq|C^+|$
and apply the same procedure as in the previous paragraph. That is, lower the coefficients on elements of $C^-$ and raise them on elements of $C^+$, all by the same amount, until some coefficients become $0$. Remove the corresponding edges from $S$. This does not increase the sum of the coefficients so it remains strictly less than $1$. The remaining edges still form a dijoin. The value of the linear combination does not change, i.e., it remains $\mathbf p$.
We may continue this until the dijoin becomes cycle-free.
\end{proof}

The following lemma is equivalent to the well-known
fact that $\tcQ_F$ is a unimodular simplex for any forest $F$, cf.\ Lemma \ref{lem:unimodular_simplex}.  

\begin{lemma}\label{lemma:interior_point_integer_coordinates}
	Let us consider a forest $F$, a positive integer $s$, and a point $\mathbf p\in s\cdot\tcQ_F$. 
    %For a forest $F$, a positive integer $s$, and a point $\mathbf p\in s\cdot\tcQ_F$, the point 
    Then $\mathbf p$ is a lattice point if and only if $\mathbf p = \sum_{e\in F} \mu_e \mathbf{x}_e$, with each $\mu_e$ an integer.
\end{lemma}

\begin{proof}[Proof of Theorem \ref{thm:degree_of_interior_poly}]
	Let $K$ be a dijoin of $G$ with cardinality $\mindj(G)$. By Lemma \ref{lem:min_dijoin_cyclefree}, this implies that the underlying undirected graph of $K$ is cycle-free.
 Then $\mathbf p = \sum_{e\in K} \mathbf{x}_e + \mathbf{0}$ is a point of $(\mindj(G)+1)\cdot\tilde{\mathcal{Q}}_G$, moreover, it clearly has integer coordinates.
	Now by Proposition \ref{prop:interior_point_iff_dijoin} we have that $\mathbf q=\frac{1}{\mindj(G)+1}\mathbf p=\sum_{e\in K} \frac{1}{\mindj(G)+1} \mathbf{x}_e$ is an interior point of $\tilde{\mathcal{Q}}_G$, 
	which implies that
	$\mathbf p$ is also an interior point of $(\mindj(G)+1)\cdot\tilde{\mathcal{Q}}_G$.
	
	We also need to prove that for $s\leq \mindj(G)$, there is no interior lattice point in $s\cdot\tilde{\mathcal{Q}}_G$. Suppose that there is an interior lattice point $\mathbf p\in s\cdot\tilde{\mathcal{Q}}_G$ for some $s\in\Z_{>0}$ and consider $\mathbf q=\frac{1}{s} \mathbf p$, 
	which is an interior point of $\tilde{\mathcal{Q}}_G$. Then by Proposition \ref{prop:interior_point_iff_dijoin} there is a cycle-free dijoin $K$ such that $\mathbf q = \sum_{e\in K} \lambda_e \mathbf{x}_e$, where $\lambda_e > 0$ for each $e\in K$, and $\sum_{e\in K}\lambda_e < 1$. 

   Now we may apply Lemma \ref{lemma:interior_point_integer_coordinates} to $s$, $K$, and $\mathbf p=\sum_{e\in K} s\lambda_e \mathbf{x}_e$. This tells us that for $\mathbf p$ to be an integer vector, $s\lambda_e$ needs to be a positive integer for each $e\in K$. Hence altogether, we have $s >\sum_{e\in K} s\lambda_e \geq |K|\geq \mindj(G)$.
\end{proof}

Next, to compute the leading coefficient, that is to prove Theorem \ref{thm:leading_coeff_graph}, we use another result (a natural continuation of Theorem \ref{thm:degree_of_h*_BR}) from Ehrhart theory. 

\begin{thm}\label{thm:leading_coeff_of_h^*}\cite[Remark 1.2]{BatNil_lattice}
  Let $P$ be a $d$-dimensional lattice polytope with $h^*$-polynomial $h^*(t)=h^*_dt^d + \dots + h^*_1t + h^*_0$ such that $h^*_d=\dots =h^*_{k+1}=0$ but $h^*_k \neq 0$. Then $h^*_k$ equals to the number of lattice points in the relative interior of $(d-k+1)P$.
\end{thm}

Hence to compute the leading coefficient of $I_G$, we need to characterize the interior lattice points in $(\nu(G)+1)\cdot \tcQ_G$. Such a result can be easily deduced from Proposition \ref{prop:interior_point_iff_dijoin} and Lemma \ref{lemma:interior_point_integer_coordinates}.

\begin{prop}\label{prop:lattice_point_char_in_dilate}
  The interior lattice points of $(\nu(G)+1)\cdot \tcQ_G$ are exactly the net degree vectors of the minimum cardinality dijoins of $G$.
\end{prop}

\begin{proof} 
By Proposition \ref{prop:interior_point_iff_dijoin}, $\mathbf{q}$ is an interior point of $(\nu(G)+1)\cdot\tcQ_G$ if and only if there is a cycle-free dijoin $K$ of $G$ such that $\mathbf{q}= \sum_{e\in K} \lambda_e \mathbf{x}_e$ where $\lambda_e > 0$ for each $e\in K$ and $\sum_{e\in E} \lambda_e < \nu(G) + 1$.

  By Lemma \ref{lemma:interior_point_integer_coordinates}, such a $\mathbf{q}$ is a lattice point if and only if each $\lambda_e$ is integer. Hence we conclude that $\lambda_e=1$ for each $e\in K$ and $|K|\leq \nu(G)$. By Lemma \ref{lem:min_dijoin_cyclefree}, minimum cardinality dijoins are all cycle-free.
  Consequently $\mathbf{q}$ is an interior lattice point if and only if $K$ is a minimum cardinality dijoin, and $\mathbf{q}$ is the net degree vector of $K$.
\end{proof}

\begin{proof}[Proof of Theorem \ref{thm:leading_coeff_graph}]
	The statement follows at once from Theorem \ref{thm:degree_of_interior_poly}, Proposition \ref{prop:lattice_point_char_in_dilate}, and Theorem \ref{thm:leading_coeff_of_h^*}.
\end{proof}

\section{Degree-minimizing orientations}
\label{sec:deg_min_orientations}

Let $\cG$ be an undirected graph. It is natural to ask about the relationship between interior polynomials of different orientations of $\cG$.

In \cite{sym_ribbon} we looked at a special case of this problem, and considered the semi-balanced orientations of a bipartite graph $\cG$ with partite classes $U$ and $W$.
Any bipartite graph $\cG$ has some (typically, many) semi-balanced orientations, but there are two special ones among them: The one where each edge is oriented from $U$ to $W$, and the one where each edge is oriented from $W$ to $U$. It is easy to see that these orientations are indeed semi-balanced. We call them the \emph{standard orientations} of $\cG$. 
The root polytopes of the two standard orientations are isometric, as they are reflections of each other.
In particular, their interior polynomials coincide.
After looking at several examples, in \cite{sym_ribbon} we 
suggested the following.

\begin{conj}
\label{conj:std_minimizes}
Among all semi-balanced orientations of a bipartite graph, the standard orientations minimize every coefficient of the interior polynomial. 
\end{conj}

See \cite[Example 6.5]{semibalanced} for some concrete instances of this phenomenon. With Theorem \ref{thm:degree_of_interior_poly} in hand, we are able to prove a weakened version of the conjecture, showing that the degree of the interior polynomial is minimized by the standard orientations. In fact, we can prove this among \emph{all} orientations.

\begin{thm}\label{thm:orientation_minimising_degree}
    Let $\mathcal{G}$ be a connected, undirected bipartite graph with partite classes $U$ and $W$. Then among all orientations of $\cG$, the degree of the interior polynomial is minimized by the standard orientations.
\end{thm}

\begin{proof}
Let $G$ be an arbitrary orientation of $\cG$. By the Lucchesi--Younger theorem \cite{LucchesiYounger}, the maximal number of disjoint directed cuts in $G$ 
is equal to $\nu(G)$, the minimal cardinality of a dijoin in $G$. 

As the degree of $I_G$ is equal to $|V(G)|-1-\nu(G)$, the degree of $I_G$ is minimized for those orientations of $\cG$ that maximize the number of disjoint directed cuts.

Let $c(\cG)$ denote the maximal number of disjoint cuts in $\cG$.
Clearly, for any orientation $G$ of $\cG$, we have $\nu(G)\leq c(\cG)$, since if we take $\nu(G)$ disjoint directed cuts in $G$, those correspond to disjoint cuts in $\cG$.

On the other hand, \cite[Theorem 9.6.12]{Frank_book} claims that the standard orientations have $c(\cG)$ disjoint directed cuts. Hence they do maximize $\nu$ among all orientations of $\cG$, as claimed.
\end{proof}

Concrete examples (e.g., \cite[Example 6.5]{semibalanced}) show that typically, there are some non-standard orientations that also minimize the degree of the interior polynomial.
It would be interesting to give a characterization for all the other orientations that attain the minimal degree.

As to the largest possible degree of the interior polynomial, Theorem \ref{thm:degree_of_interior_poly} readily implies the following.

\begin{prop}
    Among all orientations of a connected graph $\cG$, the degree of the interior polynomial is maximized by the following kinds of orientations:
    Consider the unique decomposition of $\cG$ into a block sum of bridge edges (i.e., one-element cuts) and $2$-connected graphs. Orient the bridges arbitrarily and give each $2$-connected component a strongly connected orientation.
\end{prop}
\begin{proof}
    Clearly, in any orientation, any bridge needs to be contained by any dijoin. We have a dijoin that contains only the bridge edges if and only if the $2$-connected components are strongly connected.
\end{proof}

Finally, we note that for general (non-bipartite) graphs, it can happen that there is no orientation that simultaneously minimizes all coefficients of the interior polynomial. Similarly, it can happen that there is no orientation that simultaneously maximizes all coefficients of the interior polynomial (the latter is possible also for bipartite graphs). We show two such instances in Example \ref{ex:minimizer_and_maximizer}.

It would be interesting to understand which graphs have orientations that simultaneously maximize or simultaneously minimize all coefficients of the interior polynomial.

\begin{ex}\label{ex:minimizer_and_maximizer} 
    If the graph $\cG$ is not bipartite, there is no notion of ``standard'' orientation, and there need not be any orientation coefficientwise minimizing the interior polynomial. For an example, consider the graph on the first two panels of Figure \ref{fig:min_and_max}.
    Those two orientations have interior polynomials $2x^3+6x^2+4x+1$ and $x^3+7x^2+4x+1$, respectively, whence they are incomparable, and one can check that there is no orientation that would have an interior polynomial which is coefficientwise smaller than any of them.

It is also true that sometimes there is be no orientation coefficientwise maximizing the interior polynomial. See the graph on the $3^{\text{rd}}$ and $4^{\text{th}}$ panels of Figure \ref{fig:min_and_max}. Up to isomorphism, there are $2$ strongly connected orientations here (one of them is on the $4^{\text{th}}$ panel), and they both have interior polynomial $x^5+3x^4+4x^3+4x^2+3x+1$. All other polynomials have $0$ as the coefficient of $x^5$ by Theorem \ref{thm:degree_of_interior_poly}. 
Hence our only candidate for a coefficientwise maximal interior polynomial is the one above, but since the orientation on the $3^{\text{rd}}$ panel yields $2x^4+5x^3+5x^2+3x+1$, with larger coefficients for $x^2$ and $x^3$, the candidate fails.
\end{ex}

\begin{figure}
  \begin{tikzpicture}[scale=.5]
    \tikzstyle{o}=[circle,draw,fill,scale=0.4]
    \begin{scope}[shift={(-9,0)}]
	\node [o] (0) at (2,3) {};
	\node [o] (1) at (1,4) {};
	\node [o] (2) at (2,1) {};
	\node [o] (3) at (3,2) {};
	\node [o] (4) at (0,3) {};
	\node [o] (5) at (1,0) {};	
	\node [o] (6) at (0,1) {};
	\draw [<-,>=stealth'] (0) to (1);
	\draw [->,>=stealth'] (0) to (2);
	\draw [->,>=stealth'] (0) to (3);
	\draw [->,>=stealth'] (0) to (4);
	\draw [->,>=stealth'] (1) to (4);
	\draw [->,>=stealth'] (2) to (3);
	\draw [<-,>=stealth'] (2) to (5);
	\draw [->,>=stealth'] (2) to (6);
	\draw [->,>=stealth'] (4) to (6);
	\draw [->,>=stealth'] (5) to (6);	
	\node [] at (1.5,-1.5) {\small $2x^3\!+\!6x^2\!+\!4x\!+\!1$};
  \end{scope}
  
  \begin{scope}[shift={(-3,0)}]
	\node [o] (0) at (2,3) {};
	\node [o] (1) at (1,4) {};
	\node [o] (2) at (2,1) {};
	\node [o] (3) at (3,2) {};
	\node [o] (4) at (0,3) {};
	\node [o] (5) at (1,0) {};	
	\node [o] (6) at (0,1) {};
	\draw [<-,>=stealth'] (0) to (1);
	\draw [->,>=stealth'] (0) to (2);
	\draw [->,>=stealth'] (0) to (3);
	\draw [<-,>=stealth'] (0) to (4);
	\draw [->,>=stealth'] (1) to (4);
	\draw [->,>=stealth'] (2) to (3);
	\draw [<-,>=stealth'] (2) to (5);
	\draw [->,>=stealth'] (2) to (6);
	\draw [->,>=stealth'] (4) to (6);
	\draw [->,>=stealth'] (5) to (6);	
	\node [] at (1.5,-1.5) {\small $x^3\!+\!7x^2\!+\!4x\!+\!1$};
  \end{scope}

  \begin{scope}[shift={(3.5,0)}]
	\node [o] (0) at (0,1.8) {};
	\node [o] (1) at (3,1.8) {};
	\node [o] (2) at (1.5,3.6) {};
	\node [o] (3) at (1.5,2.4) {};
	\node [o] (4) at (1.5,1.2) {};
	\node [o] (5) at (1.5,0) {};
	\draw [->,>=stealth'] (0) to (2);
	\draw [->,>=stealth'] (2) to (1);
	\draw [<-,>=stealth'] (0) to (3);
	\draw [<-,>=stealth'] (3) to (1);
	\draw [->,>=stealth'] (0) to (4);
	\draw [->,>=stealth'] (4) to (1);
	\draw [<-,>=stealth'] (0) to (5);
	\draw [<-,>=stealth'] (5) to (1);	
	\node [] at (1,-1.5) {\tiny $x^5\!+\!3x^4\!+\!4x^3\!+\!4x^2\!+\!3x\!+\!1$};
  \end{scope}
  \begin{scope}[shift={(10,0)}]
	\node [o] (0) at (0,1.8) {};
	\node [o] (1) at (3,1.8) {};
	\node [o] (2) at (1.5,3.6) {};
	\node [o] (3) at (1.5,2.4) {};
	\node [o] (4) at (1.5,1.2) {};
	\node [o] (5) at (1.5,0) {};
	\draw [->,>=stealth'] (0) to (2);
	\draw [->,>=stealth'] (2) to (1);
	\draw [<-,>=stealth'] (0) to (3);
	\draw [<-,>=stealth'] (3) to (1);
	\draw [->,>=stealth'] (0) to (4);
	\draw [->,>=stealth'] (4) to (1);
	\draw [<-,>=stealth'] (0) to (5);
	\draw [->,>=stealth'] (5) to (1);	
	\node [] at (1,-1.5) {\tiny $2x^4\!+\!5x^3\!+\!5x^2\!+\!3x\!+\!1$};
  \end{scope}
	\end{tikzpicture}
	\caption{Illustration for Example \ref{ex:minimizer_and_maximizer}.}
 \label{fig:min_and_max}
\end{figure}

\section{A generalization to regular matroids}
\label{sec:interior_regular_matroids}

It turns out that many properties of the root polytopes of digraphs extend word-by-word to oriented regular matroids, moreover, in some applications, one needs this more general case (see, for example, Section \ref{sec:greedoid_and_parking} or \cite{Eulerian_greedoid}).
Here we briefly recall the notion of a regular matroid, along with ways of orienting such a structure, and then generalize the results of Section \ref{sec:interior} to this context. For a more thorough introduction, see for example \cite{oriented_matroids_book}.

From among the many equivalent characterizations of the class of regular matroids, we will use the following: A matroid is \emph{regular} if it can be represented by the column vectors of a totally unimodular matrix. Here a matrix is \emph{totally unimodular} if each of its subdeterminants is either $0$, $-1$, or $1$. 
More concretely, let $A$ be an $n\times m$ totally unimodular 
matrix with columns $\{\mathbf a_i\}_{i\in E}$. The ground set of our matroid is going to be $E$. The rank of a set $\{i_1, \dots, i_s\}\subseteq E$ is defined as the rank of the submatrix formed by the columns $\mathbf a_{i_1}, \dots, \mathbf a_{i_s}$. A set of elements $C=\{i_1, \dots, i_s\}\subseteq E$ is a \emph{circuit} if the corresponding columns $\mathbf a_{i_1}, \dots, \mathbf a_{i_s}$ are minimally linearly dependent (over $\mathbb{R}$). 
Total unimodularity implies that in this case, the coefficients of the linear relation can be chosen from $\{-1,1\}$. For a circuit $C$, we call 
the two vectors in $\ker(A)\cap\{0,1,-1\}^n$, whose support is $C$,
the \emph{vectors of $C$}. (Minimal linear dependence implies that there are only two such vectors.)

Each circuit $C$ supports two signed circuits, both being an
ordered partition $C=C^+\sqcup C^-$, where for one of the vectors $\chi_C$ of the circuit, $C^+$ is the subset of $C$ with positive coordinates in $\chi_C$ and $C^-$ is the subset of $C$ with negative coordinates. 
For a given $C$, its two signed circuits 
differ by switching the roles of $C^+$ and $C^-$.

Let us stress that a given oriented regular matroid may have several different totally unimodular representing matrices. (For example, row elimination steps performed on $A$ do not change the oriented matroid.) On the other hand, if we multiply a column of $A$ by $-1$, the oriented matroid structure changes, while the unoriented matroid structure is preserved. (That is, the (unsigned) circuits remain the same, while the signed circuits change.) We note that for regular matroids, we do not lose any generality by concentrating on representing matrices; indeed, by \cite[Corollary~7.9.4]{oriented_matroids_book} all oriented matroid structures for a regular matroid can be obtained from a TU representing matrix.

One important example of oriented regular matroids is that of directed graphs. To a directed graph, one can associate an oriented regular matroid using its (directed) vertex-edge incidence matrix. The circuits of this matroid correspond to the cycles of the graph, and the partition $C=C^+ \sqcup C^-$ of a circuit 
is to edges of the cycle pointing in the two cyclic directions.

We call a subset $C^*\subseteq E$ a \emph{cocircuit} if there is a linear functional $h\colon\R^n\to\R$, with kernel the hyperplane $H\subseteq \mathbb{R}^n$, such that $\mathbf a_k\in H$ if $k\notin C^*$ and $\mathbf a_k\notin H$ if $k\in C^*$; moreover, $C^*$ is minimal with respect to this property. Cocircuits generalize elementary cuts of graphs. 
By the total unimodularity of $A$,
we may suppose that $h(\mathbf a_k)\in\{0,1,-1\}$ for each $k\in E$. We say that $k\in (C^*)^+$ if $h(\mathbf a_k)=1$ and $k\in (C^*)^-$ if $h(\mathbf a_k)=-1$. Again, $(C^*)^+$ and $(C^*)^-$ can switch roles, but the partition $C^*=(C^*)^+ \sqcup (C^*)^-$ is well-defined. 

A cocircuit is called \emph{directed} if either $(C^*)^+$ or $(C^*)^-$ is empty. In this case we will always suppose that $(C^*)^-$ is the empty part. 
A set $K\subseteq  E$ is called a \emph{dijoin} if $K$ intersects each directed cocircuit.

Signed circuits and signed cocircuits are orthogonal, which means that if a signed circuit $C$ and cocircuit $C^*$ have a nonempty intersection, then both $(C^+\cap (C^*)^+)\cup (C^-\cap (C^*)-)$ and $(C^+\cap (C^*)^-)\cup (C^-\cap (C^*)+)$ are nonempty \cite[Theorem 3.4.3]{oriented_matroids_book}. In particular, this means that if $C^*$ is a directed cocircuit, then if it intersects a circuit $C$, it intersects both $C^+$ and $C^-$.

We call an oriented regular matroid \emph{co-Eulerian} if $|C^+|=|C^-|$ for each circuit $C$. For graphic oriented matroids, being co-Eulerian is equivalent to the orientation of the graph being semi-balanced.

If $A$ is a totally unimodular matrix with columns $\mathbf a_1, \dots, \mathbf a_m$, then its \emph{root polytope} 
is defined as $\mathcal{Q}_A = \conv\{\mathbf a_1, \dots, \mathbf a_m\}$,
and the \emph{extended root polytope} is 
\[\tilde{\mathcal Q}_A=\conv\{\mathbf0,\mathbf a_1, \dots, \mathbf a_m\}.\]
It turns out that if $A$ and $A'$ are two totally unimodular matrices representing the same oriented matroid $M$, then the $h^*$-polynomials of $\tcQ_A$ and $\tcQ_{A'}$ are the same \cite{Eulerian_greedoid}. (We note that \cite{Eulerian_greedoid} proves this for $\mathcal{Q}_A$ and $\mathcal{Q}_{A'}$, but a straightforward modification of the argument yields the result also for $\tcQ_A$ and $\tcQ_{A'}$.)
Hence the $h^*$-polynomial of $\tcQ_A$ is an invariant of the oriented regular matroid $M$, which we call the \emph{interior polynomial}
and denote by $I_M$. Note that the orientation of the matroid matters: if we keep the (unoriented) matroid structure, but change the orientation, then the interior polynomial might change. (This is true even for graphs, cf.\ \cite[Example 6.5]{semibalanced}.) 

\begin{customthm}{\ref{thm:degree_of_interior_poly_matroids}}
	Let $M$ be a oriented regular matroid of rank $r$.
	Then the degree of $I_M$ is equal to $r-\mindj(M)$, where 
	\[\mindj(M)=\min \{|K| \mid K \subseteq  E\text{ is a dijoin of }M\}.\]
\end{customthm}

The proof proceeds through the same steps as in the graph case. First, we again need a facet description for $\tcQ_A$. 

Let us start by examining how facets of $\tcQ_A$ not containing the origin look.
Take a maximal affine independent set of vectors $\mathbf{a}_{i_1}, \dots , \mathbf{a}_{i_s}$ along such a facet. Then together with $\mathbf{0}$ they form a maximal affine independent set
among the generators of $\tcQ_A$. This happens if and only if $\mathbf{a}_{i_1}, \dots , \mathbf{a}_{i_s}$ are a maximal linearly independent set, that is, the corresponding elements form a basis in the matroid (and thus $s=r$). 
As the hyperplane of our facet does not pass through $\mathbf{0}$, we can choose a normal vector $\ell$ such that $\ell \cdot \mathbf a_{i_j} = 1$ for each $j=1, \dots, r$. 
Since $\ell \cdot \mathbf{0}=0$, this implies $\ell \cdot \mathbf a_i \leq 1$ for each $i\in  E$.

Referring to the above,
let us call a vector $\ell$ an \emph{admissible vector} if $\ell \cdot \mathbf a_i \leq 1$ for each $i\in  E$ and the elements for which equality holds form a full rank set in the matroid. 

\begin{prop}\label{prop:facet_description_matroid}
	Let $A \in\mathbb{R}^{n\times |E|}$ be a totally unimodular matrix, representing a co-Eulerian oriented regular matroid $M$. %of rank $r$.
	Let $R$ be the %affine 
	linear
	span of the columns $(\mathbf a_i)_{i\in E}$ of $A$. 
	%and $\mathbf{0}$. 
	Then the extended root polytope satisfies
	\begin{equation}\label{eq:matroid_facet_description}
		\tilde{\mathcal{Q}}_A = \left\{\,\mathbf x \in R\; \middle|
		\begin{array}{cl}
			h_{C^*}(\mathbf x) \geq 0 & \text{for all directed cocircuits }C^*\text{ of }G\\
			\ell\cdot\mathbf x \leq 1& \text{for all admissible vectors $\ell$ of $G$}
		\end{array}\right\}.
	\end{equation}
\end{prop}

Here $h_{C^*}$ is the functional that appears in the definition of a cocircuit and
we remark that our conventions 
guarantee that the restriction of 
$h_{C^*}$ 
to $R$ is determined by $C^*$. Similarly, even though the number of admissible vectors is typically infinite, there are only finitely many possibilities for the restriction of $\mathbf x\mapsto\ell\cdot\mathbf x$ to $R$.

\begin{proof}
	First note that each vector $\mathbf x\in\tcQ_A$ also belongs to the right hand side.
	It suffices to check this for the generators of the convex hull:
	For any admissible vector $\ell$, by definition $\ell \cdot \mathbf{a}_i \leq 1$ for each $i\in E$, and $\ell \cdot \mathbf{0}=0\leq1$, whence $\ell\cdot\mathbf x \leq 1$ for each $\mathbf x \in \tilde{\mathcal{Q}}_A$.
	Let $C^*$ be a directed cocircuit. Then by definition, the corresponding functional $h_{C^*}$ is such that $h_{C^*}(\mathbf a_i)=0$ for $i\notin C^*$ and $h_{C^*}(\mathbf a_i)=1>0$ for $i\in C^*$, moreover, $h_{C^*}(\mathbf{0})=0$. Hence each $\mathbf x\in \tcQ_A$ satisfies $h_{C^*}(\mathbf x)\geq 0$.
	
	Conversely, we show that any vector $\mathbf{y}$, satisfying the conditions of the right hand side, belongs to $\tcQ_A$.
	Take any facet $F$ of $\tcQ_A$. If it does not contain $\mathbf{0}$, then by the argument before the statement of the proposition, $F$ lies in the hyperplane $\{ \mathbf{a}\mid \ell \cdot \mathbf x = 1\}$ for an admissible vector $\ell$, moreover, $\tcQ_A$ is a subset of $\{\mathbf x\in\mathbb{R}^n \mid \ell\cdot\mathbf x \leq 1\}$, wherefore the hyperplane of $F$ does not separate $\mathbf y$ from $\tcQ_A$.
	
	If $F$ contains $\mathbf{0}$, then it must additionally contain a set of $r-1$ linearly independent vectors $S=\{\mathbf a_{i_1}, \dots, \mathbf a_{i_{r-1}}\}$, where $r=\dim\tcQ_A=\rank(M)$.
	Let $U\subseteq \{\mathbf a_1, \dots,\mathbf a_{|E|}\}-S$ be the set of vectors that 
	do not lie along $F$.
	Then $C^*=\{i\in  E \mid \mathbf a_i\in U\}$ is a cocircuit, with $\{\mathbf x \mid h_{C^*}(\mathbf x)=0\}$ being the linear span of $S$. 
	As $F$ is a facet, all vectors $\mathbf a_i$ must lie on one side of the span of $S$. Thus $C^*$ must be a directed cocircuit, and we also see that $\{\mathbf x \mid h_{C^*}(\mathbf x)=0\}$ supports $\tilde{\mathcal Q}_A$ along $F$. Again, we conclude that the hyperplane of $F$ does not separate $\mathbf y$ from $\tcQ_A$.
	
	Altogether, if a vector $\mathbf y$ belongs to the right hand side of \eqref{eq:matroid_facet_description}, then it is on the appropriate side of each facet-defining hyperplane, whence $\mathbf y\in\tcQ_A$.
\end{proof}

Next, we elucidate the connection of dijoins to the geometry of the extended root polytope. Analogously to the case of graphs, let us call a dijoin $K$ of a oriented regular matroid $M$ a $\emph{long-arc dijoin}$ if there is a 
signed circuit $C$ such that $|C^-|>|C^+|$ and $C^-\subseteq K$.
	
	\begin{prop}\label{prop:interior_point_iff_dijoin_matroid}
		A point $\mathbf p\in \tilde{\mathcal{Q}}_A$ is in the relative interior of $\tcQ_A$ if and only if 
		there exists a circuit-free dijoin $K$ of $M$ such that $\mathbf p = \sum_{i\in K} \lambda_i \mathbf{a}_i$, where $\lambda_i > 0$ for each $i\in K$ and $\sum_{i\in K}\lambda_i < 1$.
	\end{prop}
	
	In the proof, we will use the Farkas lemma of totally unimodular matrices (see for example \cite[Lemma 4.2.13]{Frank_book}). Let us state the lemma in the form that we will need. For two matrices $P\in\mathbb{R}^{n\times a}$ and $Q\in \mathbb{R}^{n\times b}$, we denote by $(P|Q)\in \mathbb{R}^{n\times(a+b)}$ the matrix obtained by writing the columns of $Q$ after the columns of $P$.
 
	\begin{lemma}[Farkas lemma for TU matrices]
		Suppose that $(P|Q)$ is a totally unimodular matrix, and $\mathbf b_0$, $\mathbf b_1$ are integer row vectors (of the appropriate dimension). Then there exists a row vector $\mathbf x$ such that $\mathbf x P = \mathbf b_0$, $\mathbf x Q\leq \mathbf b_1$ if and only if there do not exist column vectors $\mathbf y_0$ and $\mathbf y_1$ such that $P\mathbf y_0 + Q\mathbf y_1 =\mathbf0$, $\mathbf b_0\mathbf y_0 +\mathbf b_1\mathbf y_1 < 0$ and $\mathbf y_1\geq\mathbf 0$. Moreover, 
  if there exist $\mathbf y_0$ and $\mathbf y_1$ satisfying the inequalities, then there also exist solutions 
  $\mathbf y_0$ and $\mathbf y_1$ 
  whose coordinates are in $\{0,1,-1\}$.  
	\end{lemma}
	
	\begin{proof}[Proof of Proposition \ref{prop:interior_point_iff_dijoin_matroid}] 
		By Proposition \ref{prop:facet_description_matroid}, a point $\mathbf p\in \tcQ_A$ is in the interior of $\tilde{\mathcal{Q}}_A$ if and only if $f_{C^*}(\mathbf{p})> 0$ for each directed cocircuit $C^*$, and $\ell\cdot \mathbf{p} < 1$ for each admissible vector $\ell$. 
		Recall that the functional induced by $C^*$ satisfies 
		$f_{C^*}(\mathbf{a}_i)=0$ whenever $i\notin C^*$, and $f_{C^*}(\mathbf{a}_i)=1$ for each $i\in C^*$. 
		
		We start with the `if' direction.
		Suppose that $K$ is a dijoin, and $\mathbf{p}=\sum_{i\in K} \lambda_i \mathbf{a}_i$ with $\sum_{i\in K}\lambda_i < 1$ and $\lambda_i>0$ for each $i\in K$. Then $\mathbf{p}=\sum_{i\in K} \lambda_i \mathbf{a}_i + (1-\sum_{i\in K}\lambda_i)\cdot\mathbf{0}$, in particular $\mathbf{p}\in \tcQ_A$.
		For any directed cocircuit $C^*$, we have $f_{C^*}(\mathbf{p})=\sum_{i\in C^*\cap K} \lambda_i > 0$, since the intersection is nonempty (by the definition of a dijoin), and the summands are all positive. 
		Similarly, for any admissible vector $\ell$ we have $\ell\cdot \mathbf{a}_i\leq 1$ for any $i\in  E$, and $\ell\cdot \mathbf{0}=0$. Since the coefficient of $\mathbf{0}$ is positive, $\ell\cdot\mathbf{p} < 1$.
		This implies that $\mathbf{p}$ is in the interior of $\tcQ_A$.
		
		Now we prove the `only if' direction.
		Let $\mathbf{p}$ be an arbitrary point in the relative interior of $\tcQ_A$, and take any convex combination $\mathbf{p}=\sum_{i\in S} \lambda_i \mathbf{a}_i + \mu \cdot\mathbf{0}$ with $\lambda_i>0$ for all $i\in S$, where $S\subseteq E$ is some subset. For a point of $\tcQ_A$, we can always find such a formula (usually more than one). We claim that if $\mathbf{p}$ is in the interior, then $S$ is necessarily a dijoin. Indeed, suppose that $S$ is disjoint from a directed cocircuit $C^*$. Then $f_{C^*}(\mathbf{p})=\sum_{i\in S}\lambda_i \cdot f_{C^*}(\mathbf{a}_i)+\mu\cdot f_{C^*}(\mathbf{0})=0$, which would contradict $\mathbf{p}$ being an interior point of $\tcQ_A$.

		Next, we show that either $\sum_{i\in S} \lambda_i < 1$, or $S$ is a long-arc dijoin. 
		Then, we show that if $S$ is a long-arc dijoin, then there exist another dijoin $S'$ such that $\mathbf{p}=\sum_{i\in S'} \lambda'_i \mathbf{a}_i + \mu' \cdot\mathbf{0}$ and $\sum_{i\in S'} \lambda'_i < 1$. Lastly, we show that we can also make the dijoin circuit-free, which finishes the proof of the `only if' direction.
		
		If $\sum_{i\in S} \lambda_i < 1$, then there is nothing to prove. So suppose that $\sum_{i\in S} \lambda_i = 1$. We show that in this case $S$ is a long-arc dijoin.
		
		The fact that $\mathbf{p}$ is  an interior point of $\tcQ_A$ implies that there is no vector $\ell$ such that $\ell\cdot \mathbf{a}_i=1$ for each $i\in S$ and $\ell\cdot \mathbf{a}_i\leq 1$ for each $i\notin S$. Indeed, such a vector would give a supporting hyperplane of $\tcQ_A$ containing $\mathbf{p}$.
		
		Take the totally unimodular representing matrix $A$ of $M$. Let us gather the columns corresponding to $S$ into the submatrix $A_S$, and the columns corresponding to $E-S$ into the submatrix $A_{E-S}$. 
		
		Our previous reasoning says that there is no
		row vector $\ell$ such that $\ell \cdot A_S = \mathbf{1}$, and $\ell \cdot A_{E-S} \leq \mathbf{1}$. (Here, we denote by $\mathbf{1}$ a row vector of appropriate dimension, with all coordinates equal to 1, and with a slight abuse of notation, we do not indicate the dimension).
		
		As $A=(A_S|A_{E-S})$ is totally unimodular, by the Farkas lemma, if there exists no solution $\ell$ for the system of inequalities, then there exist $\mathbf y_0$, $\mathbf y_1$ such that $A_S\cdot \mathbf y_0 + A_{E-S}\cdot\mathbf y_1 =\mathbf 0$, $\mathbf{1}\cdot\mathbf y_0 + \mathbf{1}\cdot\mathbf y_1 < 0$, and $\mathbf y_1 \geq 0$, moreover, $\mathbf y_0$ and $\mathbf y_1$ have all their coordinates in $\{0,1,-1\}$.
		
		In other words, $\mathbf y=\begin{bmatrix}\mathbf y_0\\ \mathbf y_1\end{bmatrix}\in \ker(A)$. Recall that the circuits of $M$ are the minimal supports of elements of $\ker(A)$. 
		In this case, there exists a set of signed circuits $\mathcal{C}_0$ with mutually disjoint supports such that $\mathbf y=\sum_{C\in \mathcal{C}_0} \chi_C$ (see for example \cite[Lemma 4.1.1]{BBY}).  Hence $0 > \mathbf{1}\cdot\mathbf y = \sum_{C\in \mathcal{C}_0} \mathbf{1}\cdot \chi_C = \sum_{C\in \mathcal{C}_0} (|C^+|-|C^-|)$. This implies that there is at least one circuit $C\in \mathcal{C}_0$ such that $|C^-|>|C^+|$. As $\mathbf y_1\geq\mathbf 0$, elements of $(E-S)\cap C$ are in $C^+$.  Hence $C^-\subseteq S$. Thus, $C$ witnesses the fact that $S$ is a long-arc dijoin.
		
		Next, we show that if $\mathbf{p}=\sum_{i\in S} \lambda_i \mathbf{a}_i$ with $\lambda_i>0$, $\sum_{i\in S} \lambda_i = 1$ where $S$ is a long-arc dijoin, then we can find a different dijoin $K$ such that $\mathbf{p}=\sum_{i\in K} \lambda_i \mathbf{a}_i$ with $\lambda_i>0$ and $\sum_{i\in K} \lambda_i < 1$.
		
		Let $C$ be the circuit that has $C^-\subseteq S$ and $|C^-|>|C^+|$. Such a circuit exists because of the definition of a long-arc dijoin.
		Let $\delta = \min \{\lambda_i \mid i\in C^-\}$ and let $S'=\{ i\in C^-\mid\lambda_i=\delta\}$. Then $\mathbf p = \sum_{i\in S-S'\cup C^+} \lambda'_i \mathbf{a}_i + \mu' \cdot \mathbf{0}$ where 
		$$\lambda'_i=\left\{\begin{array}{cl} 
			\lambda_i & \text{if $i\notin C$},  \\
			\lambda_i-\delta & \text{if $i \in C^-$}, \\
			\lambda_i+\delta & \text{if $i \in C^+$},
		\end{array} \right.$$
		and $\mu'= \mu + (|C^-|-|C^+|)\cdot\delta$. This is a new convex combination for $\mathbf p$, where the coefficients are only positive for edges in $S-S'\cup C^+$, and for $\mathbf{0}$.
		Moreover, $S-S'\cup C^+$ is also a dijoin because any directed cut that intersects $C$ necessarily intersects both $C^+$ and $C^-$. Hence we have found a linear combination for $\mathbf{p}$ where now the coefficient of $\mathbf{0}$ is positive. 

        It remains to show that $S$ can be chosen circuit-free. If there is any circuit $C\subseteq S$, then orient it so that $|C^-|\geq |C^+|$. By the same argument as above, we can also write $\mathbf{p}$ as a combination of vectors of $S-S'$, where $S'$ is the set of elements of $C^-$ with minimal coefficients in the original linear combination. $S-S'$ is still a dijoin, since $C^+\subseteq S-S'$. Moreover, the coefficient of $\mathbf{0}$ can only increase by this operation. We can continue this until $S$ becomes independent in $M$, which happens in finitely many steps, as the size of the dijoin always decreases.
	\end{proof}
	
We also need the following lemma, whose proof is a literal generalization of the graph case (Lemma \ref{lem:min_dijoin_cyclefree}), which is why we omit it.
\begin{lemma}\label{lem:min_dijoin_circuitfree}
    If $C$ is a circuit and $K$ is a dijoin such that $C\subseteq K$, then $K$ is not minimal (with respect to inclusion).
\end{lemma}
	
	\begin{lemma}\label{lemma:interior_point_integer_coordinates_matroid}
		Let $F$ be an independent set in the oriented regular matroid $M$, represented by the totally unimodular matrix $A$, and let $s\in\Z_{>0}$. A point $\mathbf p\in s\cdot\tcQ_F$ is a lattice point if and only if $\mathbf p = \sum_{i\in F} \mu_i \mathbf a_i$, where each $\mu_i$ is integer.
	\end{lemma}
	
	\begin{proof}
		This follows 
		from Cramer's rule and the total unimodularity of $A$.
	\end{proof}

	\begin{proof}[Proof of Theorem \ref{thm:degree_of_interior_poly_matroids}] 
		Let $K$ be a dijoin of cardinality $\mindj(M)$. By Lemma \ref{lem:min_dijoin_circuitfree}, $K$ is independent in $M$. Then $\mathbf p = \sum_{k\in K}\mathbf a_k + \mathbf{0}$ is a point of $(\mindj(M)+1)\cdot \tcQ_A$, moreover, it has integer coordinates.
		By Proposition \ref{prop:interior_point_iff_dijoin_matroid}, 
		$\mathbf q=\frac{1}{\mindj(M)+1}\mathbf p=\sum_{k\in K} \frac{1}{\mindj(M)+1} \mathbf a_k$ is an interior point of $\tcQ_A$. Hence $\mathbf p$ is also an interior point of $(\mindj(M)+1)\cdot \tcQ_A$.
		
		We need to prove that for $s\leq \mindj(M)$ there is no interior lattice point in $s\cdot \tcQ_A$. Suppose that $\mathbf p\in s\cdot \tcQ_A$ is an interior lattice point. Take $\mathbf q=\frac{1}{s} \mathbf p \in\tcQ_A$, which is then an interior point of $\tcQ_A$. By Proposition \ref{prop:interior_point_iff_dijoin_matroid}
		there is a circuit-free dijoin $K$ such that $\mathbf q = \sum_{k\in K} \lambda_k\mathbf a_k$, $\lambda_k > 0$ for each $k\in K$, and $\sum_{k\in K}\lambda_k < 1$. In other words, $K$ is independent in $M$.
		
		We can apply Lemma \ref{lemma:interior_point_integer_coordinates_matroid} for $s$, $K$, and $\mathbf p$. This tells us that for $\mathbf p$ to be an integer vector, $s\lambda_k$ needs to be an integer for each $k\in K$. Hence $s> \sum_{k\in K} s\lambda_k \geq |K|\geq \mindj(M)$.
	\end{proof}
	
	\begin{prop}\label{prop:lattice_point_char_in_dilate_matroid}
		The interior lattice points of $(\nu(M)+1)\cdot \tcQ_A$ are exactly the vectors that can be obtained as $\sum_{i\in K}\mathbf a_i$ for some minimal cardinality dijoin $K$.
	\end{prop}
 
	\begin{proof}
		This statement can be proved the same way as Proposition \ref{prop:lattice_point_char_in_dilate}.
	\end{proof}
	
	\begin{proof}[Proof of Theorem \ref{thm:leading_coeff_matroid}]
		The statement follows from Theorems \ref{thm:degree_of_interior_poly_matroids} and \ref{thm:leading_coeff_of_h^*}, and Proposition \ref{prop:lattice_point_char_in_dilate_matroid}.
	\end{proof}
	
        One may wonder about the relationship of different orientations, 
        in particular whether some analogue of 
	Theorem \ref{thm:orientation_minimising_degree} holds. 

        \begin{defn}
            A matroid is called \emph{bipartite} if each circuit in it has even cardinality.
        \end{defn}
 
	\begin{prob}
		Given a regular matroid, or a bipartite regular matroid, which orientation has the interior polynomial of smallest degree?
		Within the class of bipartite regular matroids, is there an orientation whose interior polynomial is coefficientwise minimal? 
	\end{prob}

\section{Parking function enumerators and greedoid polynomials}
\label{sec:greedoid_and_parking}

In \cite{Eulerian_greedoid} it is proved that the parking function enumerator of an Eulerian digraph can be expressed as the interior polynomial of the cographic matroid. (Prior to that, \cite{semibalanced} settled the planar case, 
which in turn extended \cite[Corollary 5.9]{KP_Ehrhart}.)
Hence the results of the previous section give us information on the degree of the parking function enumerator of an Eulerian digraph and, equivalently, on the 
degree of the lowest term
of the greedoid polynomial. It turns out that similar results 
hold for all directed graphs.

In this section we recall the definition of a greedoid and the relationship between parking function enumerators and interior polynomials, then show how to generalize the result on the degree of the parking function enumerator to all directed graphs.

\subsection{Preliminaries on greedoids and parking functions}

Greedoids were introduced by Korte and Lov\'asz as a 
structure in which
the greedy algorithm works. 
Matroids are a special class of greedoids, but greedoids are able to express connectivity properties that matroids cannot. 

\begin{defn}[greedoid \cite{KorteLovasz}]
	A set system $\mathcal{F}$ on a finite ground set $E$ is called a \emph{greedoid} if it satisfies the following axioms.
	\begin{itemize}
		\item[(1)] $\emptyset \in \mathcal{F}$;
		\item[(2)] for all $X\in \mathcal{F} - \{\emptyset\}$ there exists $x \in X$ such that $X-x\in \mathcal{F}$;
		\item[(3)] if $X,Y\in \mathcal{F}$ and $|X|=|Y|+1$, then there exist an $x\in X-Y$ such that $Y\cup x \in \mathcal{F}$.
	\end{itemize}
	Elements of $\mathcal{F}$ are called \emph{accessible sets}, and maximal accessible sets are called \emph{bases}.
\end{defn}

It follows from the axioms that bases have the same cardinality, 
which is called the \emph{rank} of the greedoid.

An interesting subclass of greedoids is that of directed branching greedoids: 
For a digraph $G$ and its vertex $s$, the \emph{branching greedoid of $G$ rooted at $s$} is the set system consisting of the arborescences of $G$ rooted at $s$. The bases of this greedoid are the maximal arborescences. 

\begin{figure}
	\begin{tikzpicture}[scale=.20]
	\node [] (0) at (-2,10) {$s$};
	\node [circle,fill,scale=.8,draw] (1) at (0,10) {};
	\node [circle,fill,scale=.8,draw] (2) at (4,18) {};
	\node [circle,fill,scale=.8,draw] (3) at (9,14) {};
	\node [circle,fill,scale=.8,draw] (4) at (9,6) {};
	\node [circle,fill,scale=.8,draw] (5) at (4,2) {};		
	\node [circle,fill,scale=.8,draw] (6) at (16,10) {};		
	\path [thick,->,>=stealth] (1) edge [left] node {1} (2);
	\path [thick,dashed,<-,>=stealth] (1) edge [above] node {2} (3);
	\path [thick,->,>=stealth] (1) edge [below] node {3} (4);
	\path [thick,dashed,<-,>=stealth] (1) edge [left] node {4} (5);
	\path [thick,dashed,->,>=stealth] (2) edge [above] node {6} (3);
	\path [thick,<-,>=stealth] (5) edge [below] node {9} (4);
	\path [thick,->,>=stealth] (6) edge [above] node {7} (3);
	\path [thick,<-,>=stealth] (6) edge [below] node {8} (4);
	\path [thick,dashed,<-,>=stealth] (4) edge [left] node {5} (3);
	\end{tikzpicture}
	\caption{Eulerian digraph with root $s$. The non-dashed arcs form a spanning arborescence rooted at $s$.}
 \label{fig:arborescence}
\end{figure}

The greedoid polynomial was introduced by Björner, Korte, and Lovász \cite{greedoid} in several equivalent ways.
Here we recall the definition using activities with respect to a fixed ordering of the edges.

Let $B=\{b_1, \dots, b_r\}$ be a basis of the greedoid. We can form words by concatenating the elements of $B$ in some order: $b_{i_1}b_{i_2}\dots b_{i_r}$. Such a word is called \emph{feasible} if $\{b_{i_1}, \dots, b_{i_j}\}\in \mathcal{F}$ for each $j=1, \dots, r$. Note that the axioms guarantee the existence of at least one feasible word for each basis.
Let us fix an ordering of the ground set $E$. Now to any basis $B$ of the greedoid, one can associate its lexicographically minimal feasible word.

\begin{ex}
    Consider the rooted digraph of Figure \ref{fig:arborescence}, and take the ordering of the edges indicated by the labelling. The edges $\{1,3,7,8,9\}$ form a spanning arborescence, i.e., a basis of the branching greedoid. The word $39187$ is feasible for this basis, but for example the word $89713$ is not (since $\{8\}$ is not an arborescence rooted at $s$). The lexicographically minimal feasible word is $13879$.
\end{ex}

\begin{defn}[external activity for greedoids \cite{greedoid}]
	Let $(E,\mathcal F)$ be a greedoid and fix an ordering of $E$. For a basis $B$, an element $e\notin B$ is \emph{externally active} for $B$ if for any $f\in B$ such that $B \cup e - f \in \mathcal{F}$, the lexicographically minimal feasible word for $B$ is lexicographically smaller than the lexicographically minimal feasible word for $B - f \cup e$. %We call an element $e\notin B$ externally passive in $B$ if it is not externally active in $B$. 
	The \emph{external activity} of a basis $B$ is the number of externally active elements for $B$, and it is denoted by $e(B)$. 
\end{defn}

\begin{defn}[greedoid polynomial, \cite{greedoid}]
Using the above, we associate 
	$$\lambda(t) = \sum_{B: \text{ basis}} t^{e(B)}$$
to an arbitrary greedoid.
\end{defn}

We note that this is indeed well-defined, that is, independent of the ordering of the edges used to define the activities.

We will especially be interested in branching greedoids of digraphs.
Swee Hong Chan proves \cite{SweeHong_parking} that the greedoid polynomial of a branching greedoid is a simple transformation of the enumerator of graph parking functions. Let us recall these notions, too.

\begin{defn}[graph parking function]
	For a directed graph $G$ and a fixed root vertex $s$, a graph \emph{parking function} rooted at $s$ is a function $p\in \mathbb{Z}_{\geq 0}^{V-s}$ such that for each $S\subseteq V-s$, there is at least one vertex $u\in S$ with $p(u) < d(V-S,u)$, where $d(V-S,u)$ denotes the number of directed edges leading from $V-S$ to $u$.
	
	We denote the set of these functions by $\Park(G,s)$. For a parking function $p\in\Park(G,s)$, we put $|p|=\sum_{v\in V-s} p(v)$.
\end{defn}

\begin{defn}[parking function enumerator]
	For a directed graph $G=(V,E)$ and a fixed root vertex $s$, the \emph{parking function enumerator} is the polynomial $$\park_{G,s}(x) = \sum_{p \in \Park(G,s)} x^{|p|}.$$
\end{defn}

\begin{figure}
	\begin{tikzpicture}[scale=.60]
	\begin{scope}[shift={(-7.5,0)}]
	\node [] (0) at (-1,0) {$s$};
	\node [] (0) at (1.5,1.2) {$0$};
	\node [] (0) at (1.5,-1.2) {$0$};
	\node [] (0) at (2.9,0) {$0$};
	\node [circle,fill,scale=.6,draw] (1) at (-0.5,0) {};
	\node [circle,fill,scale=.6,draw] (2) at (1,1) {};
	\node [circle,fill,scale=.6,draw] (3) at (1,-1) {};
	\node [circle,fill,scale=.6,draw] (4) at (2.5,0) {};
	\path [thick,->,>=stealth] (1) edge [left] node {} (2);
	\path [thick,->,>=stealth] (1) edge [above] node {} (3);
	\path [thick,->,>=stealth] (2) edge [below] node {} (3);
	\path [thick,->,>=stealth] (3) edge [left] node {} (4);
	\path [thick,->,>=stealth] (2) edge [above] node {} (4);
	\end{scope}
	\begin{scope}[shift={(-2.5,0)}]
	\node [] (0) at (-1,0) {$s$};
	\node [] (0) at (1.5,1.2) {$0$};
	\node [] (0) at (1.5,-1.2) {$1$};
	\node [] (0) at (3,0) {$0$};
	\node [circle,fill,scale=.6,draw] (1) at (-0.5,0) {};
	\node [circle,fill,scale=.6,draw] (2) at (1,1) {};
	\node [circle,fill,scale=.6,draw] (3) at (1,-1) {};
	\node [circle,fill,scale=.6,draw] (4) at (2.5,0) {};
	\path [thick,->,>=stealth] (1) edge [left] node {} (2);
	\path [thick,->,>=stealth] (1) edge [above] node {} (3);
	\path [thick,->,>=stealth] (2) edge [below] node {} (3);
	\path [thick,->,>=stealth] (3) edge [left] node {} (4);
	\path [thick,->,>=stealth] (2) edge [above] node {} (4);
	\end{scope}
	\begin{scope}[shift={(2.5,0)}]
	\node [] (0) at (-1,0) {$s$};
	\node [] (0) at (1.5,1.2) {$0$};
	\node [] (0) at (1.5,-1.2) {$0$};
	\node [] (0) at (2.9,0) {$1$};
	\node [circle,fill,scale=.6,draw] (1) at (-0.5,0) {};
	\node [circle,fill,scale=.6,draw] (2) at (1,1) {};
	\node [circle,fill,scale=.6,draw] (3) at (1,-1) {};
	\node [circle,fill,scale=.6,draw] (4) at (2.5,0) {};
	\path [thick,->,>=stealth] (1) edge [left] node {} (2);
	\path [thick,->,>=stealth] (1) edge [above] node {} (3);
	\path [thick,->,>=stealth] (2) edge [below] node {} (3);
	\path [thick,->,>=stealth] (3) edge [left] node {} (4);
	\path [thick,->,>=stealth] (2) edge [above] node {} (4);
	\end{scope}
	\begin{scope}[shift={(7.5,0)}]
	\node [] (0) at (-1,0) {$s$};
	\node [] (0) at (1.5,1.2) {$0$};
	\node [] (0) at (1.5,-1.2) {$1$};
	\node [] (0) at (2.9,0) {$1$};
	\node [circle,fill,scale=.6,draw] (1) at (-0.5,0) {};
	\node [circle,fill,scale=.6,draw] (2) at (1,1) {};
	\node [circle,fill,scale=.6,draw] (3) at (1,-1) {};
	\node [circle,fill,scale=.6,draw] (4) at (2.5,0) {};
	\path [thick,->,>=stealth] (1) edge [left] node {} (2);
	\path [thick,->,>=stealth] (1) edge [above] node {} (3);
	\path [thick,->,>=stealth] (2) edge [below] node {} (3);
	\path [thick,->,>=stealth] (3) edge [left] node {} (4);
	\path [thick,->,>=stealth] (2) edge [above] node {} (4);
	\end{scope}
	\end{tikzpicture}
\caption{A rooted digraph (with root $s$), and its parking functions.}
\label{fig:parking_function}
\end{figure}

\begin{ex}
Figure \ref{fig:parking_function} shows a rooted digraph and each one of its parking functions. Altogether, the parking funtion enumerator is $x^2+2x+1$. 
\end{ex}

The relationship of the greedoid polynomial and the parking function enumerator is the following.

\begin{thm}\cite[Theorem 1.3]{SweeHong_parking}
	$\lambda_{G,s}(x) = x^{|E|-|V|+1} \park_{G,s}(x^{-1})$
\end{thm}

The second named author has
previously observed the following connection.

\begin{thm}\cite{Eulerian_greedoid}\label{thm:greedoid_poly_is_interior_poly_of_dual}
	Let $G$ be a connected Eulerian digraph, and let $M$ be the directed dual matroid of $G$. Then 
	\begin{equation*}
	\lambda_{G,s}(x)=
	x^{|E(G)|-|V(G)|+1} I_{M}(x^{-1}) \quad\quad \text{and}\quad\quad \park_{G,s}(x)=I_{M}(x).
	\end{equation*}
\end{thm}

Hence if $G$ is Eulerian, we can use Theorem \ref{thm:degree_of_interior_poly} to obtain a formula for the degree of the parking function enumerator, or equivalently, a formula for the 
degree of the lowest term
of the greedoid polynomial.

\begin{proof}[Proof of Theorem \ref{thm:Eulerian_degree}]
	For the (directed) cographic matroid $M$ of $G$, a dijoin of $M$ (that is, a set of edges intersecting each directed cocircuit) corresponds to an edge set of $G$ that intersects each directed cycle. Hence dijoins of $M$ correspond to feedback arc sets of $G$. Now Theorem \ref{thm:greedoid_poly_is_interior_poly_of_dual} implies the statement of the theorem.
\end{proof}

\begin{remark}
The analogue of Theorem \ref{thm:greedoid_poly_is_interior_poly_of_dual} is not true for general (non-Eulerian) digraphs. Indeed, for non-Eulerian digraphs, the parking function enumerator \emph{does} depend on the root $r$, while the interior polynomial of the dual does not.

As an example, take the digraph $G$ that has two vertices $s$ and $u$, one edge from $s$ to $u$, and two edges from $u$ to $s$.
As $G$ is planar, its dual is also a graphic matroid, associated to the digraph $G^*=(\{v_1,v_2,v_3\}, \{\overrightarrow{v_1v_2},\overrightarrow{v_2v_3},\overrightarrow{v_1v_3}\})$. 
The unique parking function of $G$ rooted at $s$ associates 0 to $u$, whereby the parking function enumerator is $\park_{G,s}(x)=1$. However, it is easy to compute that $I_{G^*}(x)=1+x$.
\end{remark}

\begin{remark}
    In \cite{Eulerian_greedoid}, % the method for proving 
    Theorem \ref{thm:greedoid_poly_is_interior_poly_of_dual} is proved by noting
    that the complements of the arborescences of $G$ (that are bases in the cographic matroid $M$ of $G$) induce a triangulation of the root polytope of $M$. If we consider a non-Eulerian rooted digraph, then we can still associate the extended root polytope of $M|_{E-F}$ to any arborescence $F$, and this will be a simplex in $\tcQ_M$. If we take these simplices for each arborescence rooted at $s$, then they will be mutually disjoint, but they will typically not fill $\tcQ_M$. In fact, their union is not even necessarily convex.
    However it is easy to check that the union covers a neighborhood of $\mathbf 0$ within $\tcQ_M$.

   Take for example the planar digraph $G$ on two vertices $s$ and $u$, with two edges from $s$ to $u$ and three edges from $u$ to $s$.
Then the dual can be given as $G^*=(\{v_1,v_2,v_3,v_4,v_5\}, \{\overrightarrow{v_1v_2},\overrightarrow{v_2v_3},\overrightarrow{v_3v_5},\overrightarrow{v_1v_4},\overrightarrow{v_4v_5}\})$. In $G$ there are two arborescences rooted at $s$, and their respective complements in $G^*$ are $\{\overrightarrow{v_1v_2},\overrightarrow{v_2v_3},\overrightarrow{v_3v_5},\overrightarrow{v_4v_5}\}$ and $\{\overrightarrow{v_1v_2},\overrightarrow{v_2v_3},\overrightarrow{v_3v_5},\overrightarrow{v_1v_4}\}$. The resulting two simplices are the convex hulls of $\mathbf 0$ and the vectors corresponding to these edges.
Now the points $\mathbf p_1=\mathbf{1}_{v_5}-\mathbf{1}_{v_4}$ and $\mathbf p_2=\mathbf{1}_{v_4}-\mathbf{1}_{v_1}$ are in the respective simplices, on the other hand it is easy to check that $(\mathbf p_1+\mathbf p_2)/2=(1/2)\cdot (\mathbf{1}_{v_5}-\mathbf{1}_{v_1})$ is not contained in either of the two simplices. That is, the union of these two simplices is not convex.
\end{remark}

Nevertheless, Theorem \ref{thm:Eulerian_degree} can still be generalized to any rooted digraph. For this, we first give a formula for the degree of the lowest term of the greedoid polynomial for general greedoids.

\subsection{General greedoids}

To give a formula for the degree of the lowest term
of an arbitrary greedoid polynomial, let us make two easy observations.

\begin{defn}
Let $X=(E,\mathcal{F})$ be a greedoid, and let $S\subseteq E$. We define $X|_S$ as the pair $(S, \mathcal{F}|_S)$, where $\mathcal{F}|_S=\{A\in\mathcal{F} \mid A\subseteq S\}$, and call it the \emph{restriction} of $X$ to $S$.
\end{defn}

\begin{claim} $X|_S$ is a greedoid.
\end{claim}
\begin{proof}
It is easy to check that the axioms are true.
\end{proof}

\begin{claim}
Let $X=(E,\mathcal{F})$ be a greedoid and $S\subseteq E$. Fix an ordering of the elements of $E$, and its restriction to $S$. 
Suppose that $F$ is a basis of both $X$ and $X|_S$. An element $e\in S-F$ is externally active for $F$ in $X|_S$ (with respect to the above mentioned ordering) if and only if $e$ is externally active for $F$ in $X$.
\end{claim}

\begin{proof}
The definition of external activity only considers bases that are subsets of $F\cup e$, and those are the same for $X|_S$ as for $X$.
\end{proof}

\begin{customthm}{\ref{thm:greedoid_nonzero_terms}}
Let $X=\{E,\mathcal{F}\}$ be a greedoid of rank $r$. Let $k=\min \{|S|\mid S\subseteq E\text{ with } \rank(X|_{E-S})=r\text{ and } \lambda_{X|_{E-S}}(0)\neq 0\}$. Then in the greedoid polynomial of $X$, the coefficient of $x^i$ is zero for $i=0, \dots, k-1$ and the coefficient of $x^k$ is nonzero.
\end{customthm}

\begin{proof}
	Take an arbitrary basis $B$ of $X$, and let $P$ be the set of elements of $E-B$ that are externally passive for $B$. We claim that $B\cup P$ is a set such that $\rank(X|_{B\cup P})=r$ and $\lambda_{X|_{B\cup P}}(0)\neq 0$. The claim
	$\rank(X|_{B\cup P})=r$ follows immediately since $B$ is a basis of $X$. On the other hand, since elements of $P$ were all externally passive for $B$ in $X$, this remains so for $X|_{B\cup P}$, thus, in $X|_{B\cup P}$ there are no externally active elements for $B$. This shows that $\lambda_{X|_{B\cup P}}(0)\neq 0$. Hence $k\leq |E-B-P|$ for any basis $B$ of $X$. As the external activity of $B$ is $e_X(B)=|E-B-P|$, this shows that if the coefficient of $x^i$ is positive in the greedoid polynomial, then $i\geq k$.
	
	It remains to show that the coefficient of $x^k$ is positive.
	Take a set $S\subseteq E$ with $|S|=k$ such that $\rank(X|_{E-S})=r$ and $\lambda_{X|_{E-S}}(0)\neq 0$. We show that there exists a basis $B$ such that $e_X(B)\leq |S|$. Since we have already proved that we cannot have $e_X(B)<k$, this will complete the proof.
	
	As $\rank(X|_{E-S})=r$, each basis of $X|_{E-S}$ is a basis of $X$. 
    Since $\lambda_{X|_{E-S}}(0)\neq 0$, the greedoid $X|_{E-S}$ has a basis $B$ such that $e_{X|_{E-S}}(B)=0$. That is, all elements of $E-S-B$ are externally passive in $B$. Hence, in $X$, only the elements of $S$ can be externally active for $B$, thus indeed, $e_X(B)\leq |S|= k$.
\end{proof}

\subsection{Arbitrary directed graphs} \label{ssec:parking_degree_nonEulerian}

In this section we use the result of the previous section to generalize Theorem \ref{thm:Eulerian_degree} to arbitrary digraphs. Let $G$ be a digraph, and let $s$ be an arbitrary fixed vertex of $G$.

Note that when examining the greedoid polynomial,
we can suppose that $G$ is $s$-root-connected, i.e., that each vertex of $G$ is reachable on a directed path from $s$. (This property is equivalent to $G$ having a spanning arborescence rooted at $s$.) Indeed, if $G$ is not $s$-root-connected, then the bases of the branching greedoid of $G$ rooted at $s$ will be the same as the bases of the branching greedoid of $G'$ rooted at $s$, where $G'$ is the subgraph of $G$ spanned by vertices reachable on a directed path from $s$. The edges of $G-G'$ will be externally semi-active for any basis, and any edge ordering, that is in this case $\park_{G}(x)=\park_{G'}(x)$ and $\lambda_{G}(x)=x^{|E(G)|-|E(G')|}\lambda_{G'}(x)$.
Hence from now on, we may suppose that $G$ is $s$-root-connected.

Swee Hong Chan \cite{SweeHong_parking} points out that in the case of directed branching greedoids, the 
notion of
greedoid activity %notion 
can be rephrased in a more graph theoretic way. For this we need some additional definitions.

Let $T$ be a spanning tree in the digraph $G$.
For an edge $e\notin T$, we say that $e'\in C(T,e)$
\emph{stands parallel} to $e$ in the fundamental cycle $C(T,e)$ if $e$ and $e'$  point in the same cyclic direction. 
Otherwise we say that $e'$ \emph{stands opposite} to $e$. (See Example \ref{ex:semi_passivity}.)

\begin{defn}[external semi-activity in digraphs, \cite{SweeHong_parking}] 
	Let $G$ be an $s$-root-connected digraph with a fixed ordering of the edges. Let $A$ be a spanning arborescence rooted at $s$ in $G$. An arc $e\notin A$ is \emph{externally semi-active} for $A$ if in the fundamental cycle $C(A,e)$ the maximal edge (with respect to the fixed ordering) stands parallel to $e$. If $e\notin A$ is not externally semi-active for $A$, then we call it \emph{externally semi-passive}.
\end{defn}

The name semi-activity comes from Li and Postnikov \cite{LiPostnikov}, who introduced this notion independently from the greedoid context.

\begin{ex}\label{ex:semi_passivity}
	For the rooted spanning arborescence in Figure \ref{fig:arborescence}, and the indicated edge ordering, edge number $2$ is externally semi-active. Indeed, its fundamental cycle is $\{2,3,8,7\}$, along which the maximal edge, $8$, stands parallel to $2$. On the other hand, $6$ is externally semi-passive, because its fundamental cycle is $\{6,1,3,8,7\}$, and among these, the maximal edge $8$ stands opposite to $6$.
\end{ex}

\begin{customthm}{\ref{thm:parking_enum_deg_general_digraph}}
	Let $G=(V,E)$ be an $s$-root-connected digraph.
	The degree of the parking function enumerator of $G$ rooted at $s$ is equal to $|E| - |V| +1 - \minfas(G,s)$.
	
	Equivalently, for the greedoid polynomial of the branching greedoid of $G$ rooted at $s$, the coefficients of $x^0, \dots, x^{k-1}$ are zero, and the coefficient of $x^{k}$ is nonzero for $k=\minfas(G,s)$. 
\end{customthm}

Recall that $\minfas(G,s)$ is the minimal cardinality of an edge set whose removal leaves an $s$-root-connected acyclic digraph (Definition \ref{def:minfas(G,s)}). It follows from Theorems \ref{thm:Eulerian_degree} and \ref{thm:parking_enum_deg_general_digraph} that for a connected Eulerian digraph and an arbitrary vertex $s$, $\minfas(G)=\minfas(G,s)$. It is also not very hard to prove this directly.

For general digraphs, $\minfas(G)$ and $\minfas(G,s)$ might differ. For example if $G$ has two vertices, $s$ and $v$, with one edge from $s$ to $v$, and two edges from $v$ to $s$, then  $\minfas(G)=1$ but $\minfas(G,s)=2$.

\begin{remark}
	For the interior polynomial (consequently, also for the parking function enumerator of Eulerian digraphs), we had a definition using Ehrhart theory. For the parking function enumerator of general digraphs we are unaware of such a definition. Hence in the proof of Theorem \ref{thm:parking_enum_deg_general_digraph}, we need to use a different, combinatorial argument.
\end{remark}

As we remarked in the introduction, Theorem \ref{thm:parking_enum_deg_general_digraph} strengthens a theorem of Björner, Korte, and Lovász, who  showed in \cite{greedoid} that for the branching greedoid of an $s$-root-connected digraph, the greedoid polynomial has a nonzero constant term if and only if the digraph is acyclic. They in fact prove more than this statement, also establishing the equivalence of this property to certain topological assumptions. For the sake of self-containedness, we give a short proof of the part of their statement that we will use.

\begin{lemma}\cite{greedoid}\label{lemma:constant_term_directed_graphs}
	Let $G$ be an $s$-root-connected digraph. The greedoid polynomial $\lambda_{G,s}$ has nonzero constant term if and only if $G$ is acyclic.
\end{lemma}

\begin{proof}
	Suppose that $G$ is acyclic and fix an arbitrary ordering of its edges. In order to show that $\lambda_{G,s}(0)\neq 0$, we need to find a spanning arborescence (rooted at $s$) with $0$ external activity. 
	
	Consider the spanning arborescence $A$ whose lexicographically minimal feasible word is lexicographically maximal. 
 For any edge $e\notin A$, if there is any arborescence of the form $A'= A \cup e - f$, then the lexicographically minimal feasible word for $A'$ can only be smaller than that of $A$. Hence it suffices to show that there is always such an arborescence $A'$. As $A$ is a spanning arborescence, it is a tree with a unique directed path from $s$ to any vertex. 
 If we add an edge $e=\overrightarrow{uv}$ to $A$, then we create a unique cycle. 
 Since we assumed that $G$ is acyclic, this cycle is not directed, in particular $v\ne s$.
 By removing the unique in-edge $f$ of $v$ that $A$ contains,
 we once again get a tree with exactly one directed path to each vertex, i.e., $A'=A \cup e - f$ is an arborescence. 
 With this, we have finished proving that $e(A)=0$.
	
	Now we show that if the $s$-root-connected graph $G$ is not acyclic, then $\lambda_{G,s}(0)=0$.
	We do this by proving that for any spanning arborescence $A$, the externally semi-active edges form a feedback arc set of $G$, whereby if $G$ is not acyclic, the set of externally semi-active edges cannot be empty for any spanning arborescence.
	
    Let $A$ be an arbitrary spanning arborescence (rooted at $s$), and let $C$ be an arbitrary directed cycle of $G$. 
    We will work in the cycle space within $\mathbb{R}^E$, which is generated by the vectors $\sum_{e\in C^+}\mathbf{1}_e -\sum_{e\in C^-} \mathbf{1}_e$ taken for all cycles $C\subseteq G$.  
    As the fundamental cycles of $A$ form a basis of the cycle space of $G$ over $\mathbb{R}$, 
    the cycle $C$ can be written (with a slight abuse of notation) as a combination of some fundamental cycles: $C=\lambda_1 \cdot C(A,e_1)+ \dots + \lambda_j \cdot C(A,e_j)$, where each coefficient is nonzero. 
    Some edges of these fundamental cycles might cancel out, but notice that $e_1, \dots, e_j$ only occur in their respective fundamental cycles, wherefore they are all part of $C$.
    We may suppose without loss of generality that each $e_i$ is positive in $C(A,e_i)$ and $C=C^+$. This implies
    $\lambda_1=\dots=\lambda_j=1$. 
    
    Let $e$ be the maximal edge in the union of the cycles
    $C(A,e_1), \dots, C(A,e_j)$. If $e\in C$, then (since $C$ is a directed cycle) $e$ is parallel to $\{e_1, \dots, e_j\}$ in $C$, and hence if $e\in C(A,e_i)$, then $e$ is also parallel to $e_i$ in $C(A,e_i)$, and it is the maximal edge in this cycle. Thus, $e_i\in C$ is externally semi-active for $A$.
    If $e\notin C$, then $e$ needs to occur in at least two fundamental cycles, once with positive, and once with negative sign. If $e$ occurs with positive sign in $C(A,e_i)$, then its maximality in this cycle ensures that $e_i \in C$ is externally semi-active for $A$. 
    In both cases we found an externally semi-active element in $C$, in particular the externally semi-active elements cover each directed cycle.
\end{proof}

\begin{proof}[Proof of Theorem \ref{thm:parking_enum_deg_general_digraph}]
	We apply Theorem \ref{thm:greedoid_nonzero_terms}. If $G$ is $s$-root-connected, then the rank of the branching greedoid rooted at $s$ is equal to $|V(G)|-1$. For an edge set $S$, the rank of the branching greedoid of $G[E-S]$ remains $|V(G)|-1$ if and only if $G[E-S]$ is $s$-root-connected. On the other hand, Lemma \ref{lemma:constant_term_directed_graphs} tells us that $\lambda_{G[E-S],s}(0)=0$ is equivalent to $G[E-S]$ being acyclic. Hence the condition of Theorem \ref{thm:greedoid_nonzero_terms} indeed gives the condition of Theorem \ref{thm:parking_enum_deg_general_digraph} for branching greedoids of root-connected digraphs.
\end{proof}

\bibliographystyle{plain}
\bibliography{Bernardi}

\begin{thebibliography}{10}

\bibitem{BBY}
Spencer Backman, Matthew Baker, and Chi~Ho Yuen.
\newblock Geometric bijections for regular matroids, zonotopes, and {E}hrhart
  theory.
\newblock {\em Forum Math. Sigma}, 7:Paper No. e45, 37, 2019.

\bibitem{BatNil_lattice}
Victor Batyrev and Benjamin Nill.
\newblock Multiples of lattice polytopes without interior lattice points.
\newblock {\em Mosc. Math. J.}, 7(2):195--207, 349, 2007.

\bibitem{BeckRobbins}
Matthias Beck and Sinai Robins.
\newblock {\em Computing the continuous discretely}.
\newblock Undergraduate Texts in Mathematics. Springer, New York, second
  edition, 2015.
\newblock Integer-point enumeration in polyhedra, With illustrations by David
  Austin.

\bibitem{universal}
Olivier Bernardi, Tam\'{a}s K\'{a}lm\'{a}n, and Alexander Postnikov.
\newblock Universal {T}utte polynomial.
\newblock {\em Adv. Math.}, 402:Paper No. 108355, 74, 2022.

\bibitem{oriented_matroids_book}
Anders Bj\"{o}rner, Michel Las~Vergnas, Bernd Sturmfels, Neil White, and
  G\"{u}nter~M. Ziegler.
\newblock {\em Oriented matroids}, volume~46 of {\em Encyclopedia of
  Mathematics and its Applications}.
\newblock Cambridge University Press, Cambridge, second edition, 1999.

\bibitem{greedoid}
Anders Björner, Bernhard Korte, and László Lovász.
\newblock Homotopy properties of greedoids.
\newblock {\em Advances in Applied Mathematics}, 6(4):447 -- 494, 1985.

\bibitem{SweeHong_parking}
Swee~Hong Chan.
\newblock Abelian sandpile model and biggs–merino polynomial for directed
  graphs.
\newblock {\em Journal of Combinatorial Theory, Series A}, 154:145 -- 171,
  2018.

\bibitem{DAli}
Alessio D'Al\`\i, Martina Juhnke-Kubitzke, and Melissa Koch.
\newblock On a generalization of symmetric edge polytopes to regular matroids.
\newblock 2023.
\newblock arXiv:2307.04933.

\bibitem{Frank_book}
Andr\'{a}s Frank.
\newblock {\em Connections in combinatorial optimization}, volume~38 of {\em
  Oxford Lecture Series in Mathematics and its Applications}.
\newblock Oxford University Press, Oxford, 2011.

\bibitem{smooth_Fano}
Akihiro Higashitani.
\newblock Smooth {F}ano polytopes arising from finite directed graphs.
\newblock {\em Kyoto J. Math.}, 55(3):579--592, 2015.

\bibitem{arithm_symedgepoly}
Akihiro Higashitani, Katharina Jochemko, and Mateusz Micha{\l}ek.
\newblock Arithmetic aspects of symmetric edge polytopes.
\newblock {\em Mathematika}, 65(3):763--784, 2019.

\bibitem{hiperTutte}
Tam\'{a}s K\'{a}lm\'{a}n.
\newblock A version of {T}utte's polynomial for hypergraphs.
\newblock {\em Adv. Math.}, 244:823--873, 2013.

\bibitem{KP_Ehrhart}
Tam\'{a}s K\'{a}lm\'{a}n and Alexander Postnikov.
\newblock Root polytopes, {T}utte polynomials, and a duality theorem for
  bipartite graphs.
\newblock {\em Proc. Lond. Math. Soc. (3)}, 114(3):561--588, 2017.

\bibitem{sym_ribbon}
Tam\'{a}s K\'{a}lm\'{a}n and Lilla T\'othm\'er\'esz.
\newblock Ehrhart theory of symmetric edge polytopes via ribbon structures.
\newblock {\em arXiv:2201.10501}, 2022.

\bibitem{semibalanced}
Tam\'{a}s K\'{a}lm\'{a}n and Lilla T\'{o}thm\'{e}r\'{e}sz.
\newblock Root polytopes and {J}aeger-type dissections for directed graphs.
\newblock {\em Mathematika}, 68(4):1176--1220, 2022.

\bibitem{KorteLovasz}
Bernhard Korte and L\'aszl\'o Lov\'{a}sz.
\newblock Mathematical structures underlying greedy algorithms.
\newblock In {\em Fundamentals of computation theory ({S}zeged, 1981)}, volume
  117 of {\em Lecture Notes in Comput. Sci.}, pages 205--209. Springer,
  Berlin-New York, 1981.

\bibitem{LiPostnikov}
Nan Li and Alexander Postnikov.
\newblock Slicing zonotopes.
\newblock unpublished, 2013.

\bibitem{LucchesiYounger}
Claudio~L. Lucchesi and Daniel~H. Younger.
\newblock A minimax theorem for directed graphs.
\newblock {\em J. London Math. Soc. (2)}, 17(3):369--374, 1978.

\bibitem{numata}
Yasuhide Numata, Yusuke Takahashi, and Dai Tamaki.
\newblock Faces of directed edge polytopes.
\newblock {\em Australas. J. Combin.}, 88:77--96, 2024.

\bibitem{PP16}
Kévin Perrot and Trung~Van Pham.
\newblock Chip-firing game and a partial {T}utte polynomial for {E}ulerian
  digraphs.
\newblock {\em The Electronic Journal of Combinatorics}, 23(1):P1.57, 2016.

\bibitem{postnikov-shapiro}
Alexander Postnikov and Boris Shapiro.
\newblock Trees, parking functions, syzygies, and deformations of monomial
  ideals.
\newblock {\em Trans. Amer. Math. Soc.}, 356(8):3109--3142, 2004.

\bibitem{Eulerian_greedoid}
Lilla T\'{o}thm\'{e}r\'{e}sz.
\newblock A geometric proof for the root-independence of the greedoid
  polynomial of {E}ulerian branching greedoids.
\newblock {\em J. Combin. Theory Ser. A}, 206:Paper No. 105891, 21, 2024.

\end{thebibliography}

\end{document}